\documentclass[11pt]{amsart}
\usepackage{setspace}
\usepackage{amsthm}

\usepackage[all]{xy}
\usepackage{amssymb}
\usepackage{enumerate}
\usepackage{mathrsfs}
\usepackage{epsfig}
\usepackage{graphicx}
\usepackage{subfig}
\usepackage{float}
\usepackage{epigraph}

\numberwithin{equation}{section}

\newtheorem{thm}{Theorem}[section]
\newtheorem{cor}[thm]{Corollary}
\newtheorem{lem}[thm]{Lemma}
\newtheorem{prop}[thm]{Proposition}
\newtheorem{rem}{Remark}[section]
\newtheorem{example}[thm]{Example}

\newtheorem{defin}[thm]{Definition}

\newcommand{\eq}[1]{(\ref{#1})}

\renewcommand{\Re}{\operatorname{\rm Re}}
\renewcommand{\Im}{\operatorname{\rm Im}}

\newcommand{\beqast}{\begin{eqnarray*}}
\newcommand{\eqast}{\end{eqnarray*}}
\newcommand{\beqa}{\begin{eqnarray}}
\newcommand{\eqa}{\end{eqnarray}}

\newcommand{\bbe}{\begin{equation}}
\newcommand{\ee}{\end{equation}}

\renewcommand{\Re}{\operatorname{\rm Re}}
\renewcommand{\Im}{\operatorname{\rm Im}}

\newcommand{\bC}{{\mathbb C}}
\newcommand{\bE}{{\mathbb E}}

\newcommand{\bQ}{{\mathbb Q}}
\newcommand{\bV}{{\mathbb V}}

\newcommand{\bR}{{\mathbb R}}

\newcommand{\bZ}{{\mathbb Z}}

\newcommand{\cB}{{\mathcal B}}

\newcommand{\cF}{{\mathcal F}}
\newcommand{\cH}{{\mathcal H}}
\newcommand{\cD}{{\mathcal D}}
\newcommand{\cE}{{\mathcal E}}
\newcommand{\cG}{{\mathcal G}}

\newcommand{\cS}{{\mathcal S}}

\newcommand{\cL}{{\mathcal L}}

\newcommand{\cC}{{\mathcal C}}

\newcommand{\cU}{{\mathcal U}}
\newcommand{\cT}{{\mathcal T}}

\newcommand{\barX}{{\bar X}}
\newcommand{\uX}{{\underline X}}

\newcommand{\cEq}{{\mathcal E_q}}
\newcommand{\cEpq}{{\mathcal E^+_q}}
\newcommand{\cEmq}{{\mathcal E^-_q}}

\newcommand{\phipq}{{\phi^+_q}}
\newcommand{\phimq}{{\phi^-_q}}

\newcommand{\tV}{{\tilde V}}

\newcommand{\hG}{{\hat G}}

\newcommand{\Om}{{\Omega}}

\newcommand{\al}{\alpha}

\newcommand{\be}{\beta}

\newcommand{\de}{\delta}
\newcommand{\eps}{\epsilon}
\newcommand{\ka}{\kappa}
\newcommand{\la}{\lambda}
\newcommand{\lp}{\lambda_+}
\newcommand{\lm}{\lambda_-}
\newcommand{\La}{\Lambda}
\newcommand{\mum}{\mu_-}
\newcommand{\mup}{\mu_+}
\newcommand{\mumpr}{\mu'_-}
\newcommand{\muppr}{\mu'_+}

\newcommand{\sg}{\sigma}

\newcommand{\om}{\omega}
\newcommand{\omm}{\om_-}
\newcommand{\omp}{\om_+}

\newcommand{\ze}{\zeta}

\newcommand{\ga}{\gamma}
\newcommand{\gap}{\gamma_+}
\newcommand{\gam}{\gamma_-}

\newcommand{\Ga}{\Gamma}

\newcommand{\barnu}{\bar\nu}
\newcommand{\barDe}{{\bar\Delta}}

\newcommand{\dd}{\partial}

\newcommand{\bfo}{{\bf 1}}

\newcommand{\supp}{{\mathrm{supp}}}

\begin{document}

\title[Efficient  inverse $Z$-transform: sufficient conditions]
{Efficient inverse $Z$-transform: sufficient conditions}
\author[
Svetlana Boyarchenko and
Sergei Levendorski\u{i}]
{
Svetlana Boyarchenko and
Sergei Levendorski\u{i}}

\begin{abstract}
 We derive several sets of sufficient conditions for applicability of the new efficient numerical realization of
 the inverse $Z$-transform. For large $n$, the complexity of the new scheme is dozens of times
 smaller than the complexity of the trapezoid rule. As applications, pricing of European options and single barrier options with discrete monitoring 
 are considered; applications to more general options with barrier-lookback features are outlined. In the case of sectorial transition operators, hence, for symmetric L\'evy models, the proof is straightforward. In the case of non-symmetric L\'evy models,
 we construct a non-linear deformation of the dual space, which makes
the transition operator sectorial, with an arbitrary small opening angle, and justify the new realization. We impose mild conditions which are satisfied 
for wide classes of
 non-symmetric Stieltjes-L\'evy processes.

\end{abstract}


\thanks{
\emph{S.B.:} Department of Economics, The
University of Texas at Austin, 2225 Speedway Stop C3100, Austin,
TX 78712--0301, {\tt sboyarch@utexas.edu} \\
\emph{S.L.:}
Calico Science Consulting. Austin, TX.
 Email address: {\tt
levendorskii@gmail.com}}

\maketitle

\noindent
{\sc Key words:} $Z$-transform, European options, barrier options, lookback options, discrete monitoring,
random walks,
L\'evy processes,  SINH-regular L\'evy processes, Stieltjes-L\'evy processes, KoBoL, NIG, trapezoid rule, sinh-acceleration

\noindent
{\sc MSC2020 codes:} 60-08,42A38,42B10,44A10,65R10,65G51,91G20,91G60

\tableofcontents

\section{Introduction}\label{s:intro}
A well-known popular method for the evaluation of the terms of a series $\{V_n\}_{n=0}^\infty$ is the discrete
Laplace transform ($Z$-transform). For the evaluation of probability distributions, which is
the main topic of interest for us, it is convenient to use the equivalent transform
\bbe\label{Vnze}
\tV(q)=\sum_{n=0}^{\infty}q^nV_n.
\ee 
If there exist $C,R>0$ such that $|V_n|\le C R^n$, then $\tV$ is analytic in an open disc $\cD(0,1/R)$ 
of radius $1/R$ centered at the origin, and
$V_n$ can be recovered using the Cauchy residue theorem 
\bbe\label{izeT0}
V_n=\frac{1}{2\pi i}\int_{|q|=r}q^{-n-1}\tV(q)dq,
\ee
where $r\in (0,1/R)$ is arbitrary. Usually, one evaluates the RHS of \eq{izeT0} applying the trapezoid rule.
The discretization error admits an upper bound via $C(r,\rho,n)\rho^{-N}$, where $\rho>1$, hence, the error decays exponentially as a function of the number $N$ of terms in the trapezoid rule. However, either $\rho$ is very close to 1, or
$C(r,\rho,n)$ is very large for $n$ large.  Therefore, if $n$ is large, one is forced to use a large $N$ to satisfy even a moderate error tolerance.  

 In \cite{EfficientDiscExtremum}, we suggested to alleviate this problem deforming the contour of integration $\{q=re^{i\varphi}\ |\ -\pi<\varphi<\pi\}$
  in \eq{izeT0} into a contour of the form
$\cL_{L; \sg_\ell,b_\ell,\om_\ell}=\chi_{L; \sg_\ell,b_\ell,\om_\ell}(\bR)$, where the conformal map $\chi_{L; \sg_\ell,b_\ell,\om_\ell}$
is of the form
 \bbe\label{eq:sinhLapl}
\chi_{L; \sg_\ell,b_\ell,\om_\ell}(y)=\sg_\ell +i b_\ell\sinh(i\om_\ell+y);
\ee
in \eq{eq:sinhLapl}, $b_\ell>0$, $\sg_\ell\in\bR$ and $\om_\ell\in (-\pi/2,\pi/2)$. 
After the deformation, we make the corresponding change of variables ({\em the sinh-acceleration})
 \bbe\label{izeT0sinh}
V_n=\int_{\bR}\frac{b_\ell}{2\pi }\chi_{L; \sg_\ell,b_\ell,\om_\ell}(y)^{-n-1}\cosh(i\om_\ell+y)\tV(\chi_{L; \sg_\ell,b_\ell,\om_\ell}(y))dy,
\ee
denote by $f_n(y)$ be the integrand on the RHS of \eq{izeT0sinh}, apply the infinite trapezoid rule
\bbe\label{Vn_inf_sinh}
V_n\approx \ze_\ell \sum_{j\in \bZ}f_n(j\ze_\ell),
\ee
and truncate the sum.
Numerical examples in \cite{EfficientDiscExtremum} demonstrate the advantages of the sinh-acceleration vs trapezoid rule
in applications to the  evaluation of the joint cumulative probability distribution function (cpdf) of a L\'evy process and its extremum.
In  the present paper, we explain how to apply the new realization of the inverse $Z$-transform to more general situations
than considered in  \cite{EfficientDiscExtremum}.

In Section \ref{s:trap_vs_sinh}, we derive general prescriptions for the choice of the parameters
of the sinh-acceleration under a general condition on the domain of analyticity $\cU$ of $\tV$ and the rate of growth of $|\tV(q)|$ as $q\to \infty$ in $\cU$, and compare the complexity of the standard numerical $Z$-inversion procedure based on the simplified trapezoid rule and the scheme utilizing the sinh-acceleration. For the error tolerance E-15 and $n$ of the order of several thousand, which is the case
for options of long maturities with daily monitoring,  the complexity of the trapezoid rule is dozens of times larger than the one of the scheme using the sinh-acceleration. In Section \ref{s:oper_form}, we consider evaluation of powers of a bounded operator $P$ acting in a Banach space assuming that
the spectrum $\sg(P)$ of $P$ is a subset of the closure $(\cC_\ga)^c$ of
$\cC_\ga:=\{\rho e^{i\varphi}\ |\ \rho>0, \varphi\in (-\ga,\ga)\}$, where $\ga\in (0,\pi/2)$ (thus, $P$ is sectorial with the opening angle $2\ga$).
As applications, we consider evaluations of expectations in symmetric L\'evy models.
In Section \ref{s:Euro_symm}, we derive explicit integral representations for prices of European options in symmetric L\'evy models.
Under additional conditions on the joint domain of analyticity and behavior at infinity of the characteristic exponent $\psi$ and  
the Fourier transform of the payoff function, an efficient algorithm is formulated. These conditions are satisfied for standard payoffs and (symmetric) sinh-regular L\'evy processes.
It is shown in \cite{SINHregular} and \cite{EfficientAmenable} that all popular classes of L\'evy processes are sinh-regular;
the symmetry condition imposes additional conditions on the parameters of the model. 

The symmetry condition implies that not only the infinitesimal generator of the process is sectorial but the transition operator
$P_t$ is sectorial as well.  In Section \ref{s:Euro_Levy_non-symmetric}, we consider pricing European options in non-symmetric L\'evy models. Let $\cF$ be the Fourier transform. Assuming that $\psi$ admits analytic continuation to the complex plane with two cuts along the imaginary axis, and the first components of almost all trajectories  of the vector field $\bV=(\dd_y\Im\psi(x+iy), -\dd_x\Im\psi(x+iy))$ on  $\bR_x\times(\bR_y\setminus\{0\})$ tend to $\infty$, we construct a non-linear deformation of the dual space which makes $\cF P_t\cF^{-1}$ 
 sectorial, with an arbitrary small opening angle. This trick allows us to prove that the sinh-deformation of the contour in the $Z$-inversion formula is applicable. The key property of the trajectories of $\bV$ is established
  for wide classes of
 non-symmetric Stieltjes-L\'evy models (SL-processes). The class of SL-processes is introduced in \cite{EfficientAmenable},
  and the general definition in terms of Stieltjes measures is the key ingredient of the proof. It is demonstrated in \cite{EfficientAmenable} that essentially all popular classes of L\'evy processes are SL-processes. 
  In Section \ref{s:barrier_Levy_non_symm}, we use appropriate pairs of deformations of $\bC_\xi-$ and $\bC_\eta-$ spaces in the pricing formula
  for single barrier options to justify the applicability of the efficient inverse $Z$-transform.
 In  Section \ref{s:concl}, we summarize the results of the paper and explain how the methodology of the paper can be used
 to price more general options with barrier-lookback features and double barrier options. Technical details are relegated to the appendix. 

\section{Trapezoid rule vs sinh-acceleration}\label{s:trap_vs_sinh}
To simplify the study of the errors of the numerical methods below, we assume that $R=1$.

\subsection{Trapezoid rule}\label{ss:trap_rule}  Introduce $h(q)=h(r,q)=(qr)^{-n}\tV(qr)$ and rewrite \eq{izeT0} as
\bbe\label{izeT}
V_n=\frac{1}{2\pi i}\int_{|q|=1}h(q)\frac{dq}{q},\ n=0,1,2,\ldots
\ee 
  Denote the RHS  of \eq{izeT} by $I(h)$. Usually, one approximates $I(h)$ with
  \bbe\label{defTM}
T_N(h) = (1/N) \sum_{k=0}^{N-1} h(\zeta_N^k),
\ee
where $N>1$ is an integer, and $\zeta_N=\exp(2\pi i/N)$ is the standard primitive $N$-th root of unity.
For $0<a<b$, denote by $\cD_{(a,b)}$ the annulus $\{q\in \bC\ |\ a<|q|<b\}$. For any $\rho\in (1, 1/r)$,
$h(z)$ is analytic in $\cD_{(1/\rho,\rho)}$.
The Hardy norm of $h$ is
\[
\|h\|_{\cD_{(1/\rho,\rho)}}=\frac{1}{2\pi i}\int_{|z|=1/\rho}|h(z)|\frac{dz}{z}+\frac{1}{2\pi i}\int_{|z|=\rho}|h(z)|\frac{dz}{z}.
\]
The error bound is well-known. See e.g, \cite{TrefethenWeidmanTrapezoid14}.
\begin{thm}\label{disctraperror}
Let $h$ be analytic in  $\cD_{(1/\rho,\rho)}$, where $\rho>1$.  
The error of the trapezoid approximation admits the bound
\bbe\label{errtrapgen}
|T_N(h)-I(h)|\le \frac{\rho^{-N}}{1-\rho^{-N}}\|h\|_{\cD_{(1/\rho,\rho)}}.
\ee
\end{thm} To satisfy a small error tolerance $\eps>0$, it is necessary to choose $N=N(\eps,n)$ and $\rho$ so that $\rho^{-N}$ is small, hence,
we may use an approximate bound
\bbe\label{errtrapgen_app}
|T_N(h)-I(h)|\le \rho^{-N}\|h\|_{\cD_{(1/\rho,\rho)}}.
\ee
The reader observes that to satisfy a small error tolerance $\eps$ with a moderate $N$, it is necessary to
choose $r$ so that $\rho\in (1,1/r)$ can be chosen not small. But if $1/\rho$ is small and 
  $n$ is large, then the Hardy norm $\|h\|_{\cD_{(1/\rho,\rho)}}$ is very large, hence, one is forced to use
  a large $N$ to satisfy even a moderate error tolerance $\eps$. 
  The bound \eq{Neps_n_M} below is derived under a realistic assumption that one can  evaluate the terms in
  the trapezoid rule sufficiently accurately only if the terms are not too large. We impose the condition
  on the admissible size of the terms in the form $r^{-n}\le e^M$, and derive an approximation $N_{appr}=N_{appr}(\eps,n,M)$ to $N=N(\eps,n,M)$ in terms of $E=\ln(1/\eps), n$ and $M$. \footnote{We write $N(\eps,n,M)\approx N_{appr}(\eps,n,M)$ if there exist $c,C>0$ independent of $(\eps,n,M)$
  such that $cN_{appr}(\eps,n,M)\le N(\eps,n,M)\le CN_{appr}(\eps,n,M)$.}
  It is seen from \eq{Neps_n_M} that  $N$ decreases as $M$ increases. 
    \begin{lem}\label{lem:lower_bound_comp_trap}
  Let there exist $C_0>c_0>0$ such that 
  \bbe\label{eq:bound_main_trap}
  c_0|1-q|^{-1}\le |\tV(q)|\le C_0|1-q|^{-1}, \ q\in \cD(0,1),
  \ee
  and let $r^{-n}= e^M$, where $M$ is independent of $n$. 
  Then, if $n>>1$, $\eps<<1,$ and $E/n>>1$,
  \bbe\label{Neps_n_M} 
  N(=N(\eps,n,M))\approx \frac{n}{M}(E+2M).
  \ee
 
  \end{lem}
  
  \begin{proof} 
 If $n\to\infty$, we must have $r_-:=r/\rho\to 1$ and $r_+:=r\rho\to 1$. Therefore, it follows from \eq{eq:bound_main_trap}
that
 there exist $C_1,c_1>0$ independent of $r_\pm$ such that 
  \bbe\label{eq:bound_H_trap_0}
c_1(r_-^{-n}H(r_-)+r_+^{-n}H(r_+))\le \|h\|_{\cD_{(1/\rho,\rho)}}\le C_1(r_-^{-n}H(r_-)+r_+^{-n}H(r_+)),
\ee
where 
\bbe\label{eqpm:H_trap}
H(r_\pm) :=\int_{-\pi}^\pi\left|1-r_\pm e^{i\varphi}\right|^{-1}  d\varphi\sim -2\ln(1-r_\pm)
\ee
(see Section \ref{Proof of eq:bound_H_trap} for the proof). 
It follows from \eq{eq:bound_H_trap_0}, \eq{eqpm:H_trap}
and \eq{errtrapgen_app} that if $\rho^{-N}<<1$, then
\bbe\label{eq:Nasymp_trap}
N\sim\frac{E+\ln[2r_-^{-n}(-\ln(1-r_-))+2r_+^{-n}(-\ln(1-r_+))]}{\ln\rho}.
\ee
We represent $\rho$ in the form $\rho=e^{M_1/n}$. Since $r\rho<1$, $M_1<M$. We have $r_-^{-n}=e^{M+M_1}$ and $r_+^{-n} =e^{M-M_1}$, 
therefore, as
  $n\to\infty$, 
$-\ln(1-r_-)=(M+M_1)/n+O(M^2/n^2), -\ln(1-r_+)=(M-M_1)/n+O(M^2/n^2)$. We take $M_1$ very close to $M$ but not very close so that  the second term in the square brackets in the numerator on the RHS of
\eq{eq:Nasymp_trap} does not exceed the first one. We also assume that $n>> E>> \ln n$ and $E>>\ln M$. Then the approximation \eq{Neps_n_M} holds.
  
  \end{proof}

\subsection{Sinh-acceleration}\label{ss:sinh} 
We justify the sinh-deformation under the following condition. 
 \vskip0.1cm
 \noindent
 {\sc Condition $Z$-SINH$(\ga)$.}  
 There exists $\ga\in (0,\pi)$  such that
 \begin{enumerate}[(a)]
 \item $\tV(q)$ admits analytic continuation
 to $(-\cC_{\pi-\ga})\cup \cD(0, 1)$;
 \item
 for any $r\in (0,1)$ and $\ga_1\in (\ga,\pi)$, there exist $C_\tV=C_\tV(r,\ga_1)$ and $a_\tV=a_\tV(r,\ga)$ such that
 \bbe\label{eq:main_bound}
|\tV(q)|\le C_\tV(1+|q|)^{a_\tV}, \
q\in (-\cC_{\pi-\ga})\cup \cD(0, r).
\ee
\end{enumerate} 
\begin{rem}\label{rem:CtV,atV}{\rm 
In application that we consider, $\ga< \pi/2$, and $a_\tV$ is not large and independent of $r,\ga$.
  In the case of pricing options with discrete monitoring in wide classes of L\'evy models,
$Z$-SINH($\ga$) holds with an arbitrary small positive $\ga$.

}
\end{rem}
For $\mum<\mup$, set $S_{(\mum,\mup)}=\{z\in\bC\ |\ \Im z\in (\mum,\mup)\}$. 
\begin{lem}\label{lem:Z-SINH}
Let Condition $Z$-SINH($\ga$) hold with $\ga\in (0,\pi/2)$. 

Then, for any $r_+\in (0,1)$, $\om_\ell\in (\ga/2, \pi/4)$ and
$d_\ell>0$ such that $\om_\ell-d_\ell>\ga/2$ and $\om_\ell+d_\ell<\pi/4$, there exists $r_{0,-}\in (0,r_+)$
such that for any $r_-\in (r_{0,-},r_+)$,
\begin{enumerate}[(a)]
\item  there exist unique $\sg_\ell\in \bR$ and $b_\ell>0$ such that
$\sg_\ell-b_\ell\sin(\om_\ell-d_\ell)=r_+$,  $\sg_\ell-b_\ell\sin(\om_\ell+d_\ell)=r_-$, and
$r:=\sg_\ell-b_\ell\sin(\om_\ell)\in (r_-,r_+)$;
\item
the curves $\cL_{L; \sg_\ell,b_\ell,\om_\ell}$ and $\cL_{L; \sg_\ell,b_\ell,\om_\ell\pm d_\ell}$ pass through points $r$ and $r_\pm$,
respectively;
\item
$\chi_{L; \sg_\ell,b_\ell,\om_\ell}(S_{(-d_\ell,d_\ell)})\subset (-\cC_{\pi-\ga})\cup \cD(0, 1)$;
\item
$\cL_{L; \sg_\ell,b_\ell,\om_\ell\pm d_\ell}$ are the  boundaries of $\chi_{L;\sg_\ell,b_\ell,\om_\ell}(S_{(-d_\ell,d_\ell)})$;
\item
the distance of the left boundary $\cL_{L; \sg_\ell,b_\ell,\om_\ell+ d_\ell}$ from the origin is $r_-$;
\item if $n>a_\tV$, the deformation of the contour $\{|q|=r\}$ into $\cL_{L; \sg_\ell,b_\ell,\om_\ell}$ 
is justified.
\end{enumerate}
\end{lem}
\begin{proof} First, note that for any $0<r_-<r_+<1$, (a)-(b) are satisfied with \beqa\nonumber
\sg_\ell&=&\frac{r_+\sin(\om_\ell+d_\ell)-r_-\sin(\om_\ell-d_\ell)}{\sin(\om_\ell+d_\ell)-\sin(\om_\ell-d_\ell)}
\\\label{sgll_sinh_L}&=&\frac{(r_+-r_-)\sin(\om_\ell)\cos(d_\ell)+(r_++r_-)\cos(\om_\ell)\sin(d_\ell)}{2\cos(\om_\ell)\sin(d_\ell)},
\\\label{bell_sinh_L}
b_\ell&=&\frac{r_+-r_-}{\sin(\om_\ell+d_\ell)-\sin(\om_\ell-d_\ell)}=\frac{r_+-r_-}{2\cos(\om_\ell)\sin(d_\ell)}.
\eqa 
(c)-(d)
As $r_-\to r_+$, $\sg_\ell\to r_+$ and $b_\ell\to 0$. If $r_+-r_-\to 0$ and $q=q(y)\in\cL_{L; \sg_\ell,b_\ell,\om_\ell-d_\ell}$ 
is close to $-e^{i(\pi-\ga)}$ in absolute value, then $|y|$ tends to infinity. It follows that 
$
q(y)\sim r_++(b_\ell/2)\exp[i(\om_\ell-d_\ell)+y]$. Hence, if $r_+-r_-$ is sufficiently small, the curve $\cL_{L; \sg_\ell,b_\ell,\om_\ell-d_\ell}$ is to the left of
$-e^{\pm i(\pi-\ga)}$, which implies that $\cL_{L; \sg_\ell,b_\ell,\om_\ell+d_\ell}\subset (-\cC_{\pi-\ga})\cup \cD(0,r_+)$. It remains to note that the curve $\cL_{L; \sg_\ell,b_\ell,\om_\ell+d_\ell}$ is to the left of 
the curve $\cL_{L; \sg_\ell,b_\ell,\om_\ell-d_\ell}$, the former (the latter) curve being the left (right) boundary of 
the image of $S_{(-d_\ell,d_\ell)}$ under $\chi_{L; \sg_\ell,b_\ell,\om_\ell}$.

(e) The curve $\cL_{L; \sg_\ell,b_\ell,\om_\ell+d_\ell}$ is smooth, symmetric w.r.t. the real axis, passes through $r_->0$ and  
the rays $e^{\pm i(\pi/2+\om_\ell+d_\ell)}$ are asymptotes of $\cL_{L; \sg_\ell,b_\ell,\om_\ell+d_\ell}$. It remains to note that the region to the left from $\cL_{L; \sg_\ell,b_\ell,\om_\ell+d_\ell}$
is convex and  $\pi/2+\om_\ell+d_\ell<3\pi/4$.

(f) 
In the process of deformation, the contour of integration remains in the domain of analyticity, stabilizes to rays at infinity, and \eq{eq:main_bound}
holds. 

\end{proof}

\begin{rem}\label{rem:om_ell_de_ell}{\rm
\begin{enumerate}[(a)]
\item
If $\ga\in(0,\pi/2)$ is not very small, and $r_+$ is very close to 1, it is approximately optimal to choose $\om_\ell=\ga/4+\pi/8$ and $d_\ell<d_0:=\pi/8-\ga/4$ close to $d_0$, e.g.,
$d_\ell=k_d d_0$, where $k_d\in [0.8,0.9]$. 
\item
If $\ga> 0$ is very small and $n$ is not very large so that we can choose $r_+<\cos(\ga)$,
then it is possible and may be advantageous to choose $\om_\ell\in (\ga/2-\pi/2,0)$. See \cite{EfficientDiscExtremum} for examples.
An approximately optimal choice is $\om_\ell=\ga/2-\pi/8$, 
$d_\ell=3\pi/8-\ga/2$.
\end{enumerate}

}\end{rem}
 After the deformation, we make the corresponding change of variables \eq{izeT0sinh},
denote by $f_n(y)$  the integrand on the RHS of \eq{izeT0sinh}, and apply the infinite trapezoid rule. The result is \eq{Vn_inf_sinh}.
A bound for the discretization error of the infinite trapezoid rule, hence, recommendation for the choice of $\ze_\ell$ sufficient to satisfy a given error tolerance $\eps$,
is easy to derive because 
 $f_n$  is analytic in a strip $S_{(-d_\ell,d_\ell)}$, and  \eq{eq:main_bound} holds. 
We have $\lim_{R\to \pm\infty}\int_{-d_\ell}^{d_\ell} |f_n(i s+R)|ds=0,$
and 
\bbe\label{Hnorm}
H(f_n,d_\ell):=\|f_n\|_{H^1(S_{(-d_\ell,d_\ell)})}:=\lim_{s\downarrow -d_\ell}\int_\bR|f_n(i s+ t)|dt
+\lim_{s\uparrow d_\ell}\int_\bR|f_n(i s+t)|dt<\infty.
\ee
As in \cite{feng-linetsky08}, we call $H(f_n,d_\ell)$ the Hardy norm; the standard definition is marginally different.
The following key lemma for functions analytic in a strip is proved in \cite{stenger-book} using the heavy machinery of sinc-functions. A simple proof 
can be found in \cite{paraHeston}.
\begin{lem}[\cite{stenger-book}, Thm.3.2.1] 
Let $f_n$ satisfy the conditions above. Then
\bbe\label{Err_inf_trap}
\left|V_n- \ze_\ell \sum_{j\in \bZ}f_n(j\ze_\ell)\right|\le H(f_n,d_\ell)\frac{\exp[-2\pi d_\ell/\ze_\ell]}{1-\exp[-2\pi d_\ell/\ze_\ell]}.
\ee
\end{lem}
 Once
an approximate bound $H_{\mathrm{appr.}}(f_n,d)$ for $H(f_n,d)$ is derived, it becomes possible to satisfy the desired error tolerance $\eps$
with a good accuracy letting
\bbe\label{rec_ze_ze}
\ze_\ell=2\pi d_\ell/\ln(H_{\mathrm{appr.}}(f_n,d)/\eps).
\ee
 We truncate the series on the RHS of \eq{Vn_inf_sinh}
\bbe\label{Vn_inf_sinh_trunc}
V_n\approx \ze_\ell \sum_{|j|\le N_0}f_n(j\ze_\ell).
\ee
Given the error tolerance  $\eps$,  
a  good approximation to the truncation parameter $\La:=N_0\ze_\ell$ is
\bbe\label{eqLa_z}
\La=\frac{1}{n-a_{\tV}}\ln\frac{C_{\tV}(r)}{\eps}-\ln\frac{b_\ell}{4\pi}+\La_0,
\ee
where $C_{\tV}$ and $a_{\tV}$ are from \eq{eq:main_bound}, and $\La_0$ is the length of the intersection
$\cL_{L; \sg_\ell,b_\ell,\om_\ell}\cup \cD(0,1)$. This follows from the asymptotics 
$f_n(y)=(1+o(1))(b_\ell/(4\pi))e^{-|y|}$ as $y\to\pm\infty$.

If $V_n$ are real, then $\overline{h(z)}=h(\bar z)$, and, therefore, we can simplify \eq{Vn_inf_sinh_trunc}
\bbe\label{Vn_inf_sinh_trunc_sym}
V_n\approx 2\ze_\ell \Re \sum_{j=0}^{N_0}f_n(j\ze_\ell)(1-\de_{j0}/2).
\ee
 If $a_\tV$ is not large (in the applications in the paper, $a_\tV<7$), the complexity of the numerical scheme (number of the terms in the simplified trapezoid rule) is of the order of 
 \[\left(\frac{\ln(1/\eps)}{n}-\ln\frac{b_\ell}{4\pi}+\La_0\right)\frac{\ln(H(f_n,d_\ell))+\ln(1/\eps)}{2\pi d_\ell}.\]

\subsection{Parameter choice}\label{param_choice_Z_SINH} We give a detailed prescription for the parameter choice under the following condition similar to \eq{eq:bound_main_trap}:
$\ga \in (0,\pi/2)$, and  there exist $C_{\tV,\ga}>0$ and $a_\tV$ such that
 \bbe\label{eq:main_bound_2}
|\tV(q)|\le C_{\tV,\ga}|1-q|^{-1}|q|^{a_\tV}, \
q\in (-\cC_{\pi-\ga})\cup \cD(0, 1).
\ee
Note that if $n$ is small or moderate, then the sinh-deformation brings small advantages or none. Hence, we assume that $n$ is large. In this case, the Hardy norm is very large unless the interval
$[r_-,r_+]:=\chi_{L; \sg_\ell,b_\ell,\om_\ell}(S_{(-d_\ell,d_\ell)})\cap\bR\subset(0,1)$ 
is very close to 1,
and $\chi_{L; \sg_\ell,b_\ell,\om_\ell}(S_{(-d_\ell,d_\ell)})$ is at the distance $r_-$ from the origin. \footnote{If the latter property
fails, the factor $q^{-n-1}$ can be much larger in absolute value than $r_-^{-n-1}$  for $q$ on the left boundary of 
$\chi_{L; \sg_\ell,b_\ell,\om_\ell}(S_{(-d_\ell,d_\ell)})$, hence, the Hardy norm very large. } Let $M$ and $M_1$ be as in the case of the trapezoid rule.


\vskip0.1cm
\noindent
{\sc Algorithm.}  
Let $n$ be large. Given $M$,
\begin{enumerate}[(1)]
\item
choose $M_1<M$ close to $M$, e.g., $M_1=0.9M$;
\item
set $r_-=e^{-(M+M_1)/n}$ and $r_+=e^{-(M-M_1)/n}$;
\item
choose $\om_\ell$ and $d_\ell$ as in Remark \ref{rem:om_ell_de_ell};
\item
define
 $(\sg_\ell,b_\ell)$ by \eq{sgll_sinh_L}, \eq{bell_sinh_L}. If necessary, increase $\om_\ell$ and decrease
 $d_\ell$ so that\\ $\sg_\ell-b_\ell\sin(\om_\ell-d_\ell)\le r_+$ and $\sg_\ell-b_\ell\sin(\om_\ell+d_\ell)\ge r_-$;
\item
define step $\ze$ and number of terms $N_0=\La/\ze$ by \eq{rec_ze_ze} and \eq{eqLa_z}.
\end{enumerate}
To apply \eq{rec_ze_ze}, we need an efficient approximation to the Hardy norm.
The Hardy norm being an integral one, we may derive an approximate bound working in the $q$-plane. On the strength of
\eq{eq:main_bound_2}, $H(f_n,d_\ell)$ admits an approximation 
\[
H(f_n,d_\ell)\approx \frac{1}{\pi} C_{\tV,\ga}(r_+^{-n-1}H(r_+)+r_-^{-n-1}H(r_-)), \]where 
\[
H(r_\pm)=
\int_0^{+\infty}\left|1-r_\pm-te^{i(\om_\ell+\pi/2\mp d_\ell)}\right|^{-1}(1+t)^{-n-1+a_{\tV}}dt.
\]
Similarly to \eq{eqpm:H_trap}, $H(r_\pm)\sim -2\ln(1-r_\pm)$ as $n\to \infty$.
If
$r_+$ is chosen not too close to 1,  then 
$H(f_n,d_\ell)\approx (2/\pi) C_{\tV,\ga}e^{M+M_1}n/(M+M_1)$,
$\ln H(f_n,d_\ell)\approx 2M+\ln n, $
and we can use
\bbe\label{ze_ell}
\ze_\ell=\frac{2\pi d_\ell}{E+\ln n+ 2M}=\frac{k_d\pi(\pi/4-\ga/2)}{E+\ln n+ 2M}.
\ee
Assuming that  $\ga\in (0,\pi/2)$ is close to $0$, we have 
\[
b_\ell=\frac{r_+-r_-}{\sin(\om_\ell+d_\ell)-\sin(\om_\ell-d_\ell)}\approx \frac{2M}{n(1/\sqrt{2}-\sin(\ga/2))},\] and  \eq{eqLa_z} gives $\La$ satisfying
\bbe\label{eqLa_z_1}
\La\le C_{\tV,\ga}\frac{E}{n-a_{\tV}}+\ln n-\ln M,
\ee
 where $C_{\tV,\ga}$ is independent of $\eps$ and $n$. The number of terms is, approximately,
 \bbe\label{N_ell_sinh}
 N_\ell\approx\frac{E+\ln n+ 2M}{k_d\pi^2/4}\left(C_{\tV,\ga}\frac{E}{n-a_{\tV}}+\ln (n/M)+\ga\right).
 \ee
 In applications in the paper, \eq{eq:main_bound_2} holds with $a_\tV\in [0,7)$. Assume that 1) $n$ is large but not extremely large so that
 $n>> E>> \ln n$, and 
 2) $E>>M$. Then we may use the approximation
  \bbe\label{N_ell_sinh_2}
N_\ell\approx\frac{(E+2M)\ln (n/M)}{k_d\pi^2/4}.
\ee

\subsection{Comparison of the complexity of the two schemes}
Comparing \eq{Neps_n_M} with \eq{N_ell_sinh_2}, we see that the complexity of the trapezoid rule
exceeds the complexity of the new numerical realization of the inverse $Z$-transform by the factor 
$K\approx (n/M)/\ln(n/M)$.
\begin{example}{\rm
Assume (rather unrealistically) that the individual terms in the trapezoid rule can be calculated with the required accuracy 
(e.g., $\eps=10^{-16}$) even if the size of the individual terms
in the trapezoid rule is of the order of $10^{10}$.
Then $M\approx 23$. In the case of daily monitoring and time horizon $T=5Y$, $n=1260$, $n/M\approx 55$, and $K\approx 14$.
If $T=15Y$, $K\approx 32$, and if $T=30$, $K\approx 57$. 
Under more realistic assumptions, $M$ is smaller and the gain coefficient $K$ larger.
}\end{example}

\begin{rem}\label{rem: K-gain}
{\rm
\begin{enumerate}
[(a)]
\item
In the case of pricing of exotic options,
evaluation of the individual terms is time consuming. 
Hence, typically, one can calculate the individual terms with the accuracy worse than E-15, and then the gain in
speed
of the new realization of the inverse $Z$-transform is greater still. See examples in  \cite{EfficientDiscExtremum}.
\item
 The ratio of the LHS and RHS  of \eq{Neps_n_M} is close to 1, whereas the ratio
$N_\ell/((E+2M)\ln(n/M))$ is close to $4/\pi^2$. Accurate asymptotic statements are too cumbersome to be useful
but numerical experiments support the claim. Hence, if $\ga$ is very close to 0, which is the case
in the majority of applications in the present paper, than the ``gain coefficient" $K$ increases two-fold. 
\end{enumerate}
}
\end{rem}


   \section{Evaluation of operators $P^n$}\label{s:oper_form}
   \subsection{Contour deformations: sufficient conditions}
 Let $P$ be a bounded operator in the Banach space $\cB$;  $\sg(P)$ and $\rho(P)$ denote the spectrum and spectral radius of $P$. 
 Let $n\ge 1$ be an integer and $r\in (0,1/\rho(P)))$.
 Applying the residue theorem, one obtains  
\bbe\label{VB}
P^n=\frac{1}{2\pi i}\int_{|q|=r}(1-qP)^{-1}q^{-n-1}dq.
\ee
\begin{thm}\label{thm: sep from 0}
 Let  $P$ be a bounded linear operator in the Banach space $\cB$, and 
let there exist $\ga\in (0,\pi/2)$ such that $\sg(P)\subset(\cC_{\ga})^c$.
 Then
 \begin{enumerate}[(a)]
 \item
for $q\in (-\cC_{\pi-\ga})\cup\{q\ |\ |q|<1/\rho(P)\}$, $I-qP$ is invertible;
\item
  for any $r\in (0,1/\rho(P))$ and $\ga_1\in (\ga,\pi/2)$, 
there exists $C>0$ such that \bbe\label{bound_P_gen}
\|(I-qP)^{-1}\|\le C(1+|q|), \
q\in (-\cC_{\pi-\ga_1})^c\cup\cD(0,r).
\ee

\end{enumerate}
\end{thm}
\begin{proof}
(a) If $|q|<1/\rho(P)$, the invertibility of $I-qP$ is a textbook fact. 
If  $q\in (-\cC_{\pi-\ga})$, then $q^{-1}\in (-\cC_{\pi-\ga})$ as well, and since $(-\cC_{\pi-\ga})\cap (\cC_\ga)^c=\emptyset$, $I-qP$ is invertible. (b) If $q\in (-\cC_{\pi-\ga_1})^c\cup\{q\ |\ |q|\le r\}$, then
$q^{-1}\in (-\cC_{\pi-\ga_1})^c\cup\{q\ |\ |q|\ge 1/r\}$. The intersection of the latter set with $\sg(P)$ is empty, hence,
there exists $C>0$ such that for all $q\in (-\cC_{\pi-\ga_1})^c\cup\{q\ |\ |q|\le r\}$, $\|(q^{-1}-P)^{-1}\|\le C$,
and (b) follows.
 
\end{proof}
Thus, the operator $\rho(P)^{-1}P$ satisfies (the operator form of) Condition $Z$-SINH($\ga$), and we can use
Lemma \ref{lem:Z-SINH} and Remark \ref{rem:om_ell_de_ell} to choose the parameters of the sinh-deformation of
the contour in \eq{VB}. See Algorithm in Section \ref{param_choice_Z_SINH}.

If $n$ is large, the Hardy norm can be very large unless the deformed contour is as far from the origin as possible.
To choose the parameters of the deformation so that the Hardy norm is not too large, we need conditions analogous to \eq{eq:main_bound_2}.
\begin{thm}\label{thm:norm_inv_Z}
Let $P$ be a non-invertible bounded normal operator  in the Hilbert space $\cH$, and let there exist $\ga\in (0,\pi/2)$
such that $\sg(P)\subset(\cC_\ga)^c$ (or $\ga=0$ and $P$ is non-negative self-adjoint). Then 
\begin{enumerate}[(a)]
\item
for $q\in (-\cC_{\pi-\ga})\cup \cD(0,1/\|P\|)$, $I-qP$ is invertible, and 
\bbe\label{norm_bound_main}
\|(I-qP)^{-1}\|=\max_{\la\in \sg(P)}|1-q\la|^{-1};
\ee
\item
for any $\ga'\in (\ga,\pi/2)$,
\bbe\label{normal_bound_ga_pr}
\|(I-qP)^{-1}\|\le 1/\sin(\ga'-\ga),\ q\in-\cC_{\pi-\ga'}; 
\ee
\item
for any $r\in (0,1)$, there exists $C(r)$ independent of $\ga$ such that if $\sin\ga\in (0,1-r)$, then
\bbe\label{normal_bound_disc_small}
\|(I-qP)^{-1}\|\le C(r)/|1-q\|P\||, \ q\in \cD(0,(1-\sin\ga)/\|P\|);
\ee
\item
if $P\ge 0$ is self-adjoint, then 
\bbe\label{self_adjoint_bound}
\|(I-qP)^{-1}\|\le 4/|1-q\|P\||, \ q\in\cD(0,1/\|P\|).
\ee
\end{enumerate}
\end{thm}
\begin{proof} (a) is immediate from the spectral decomposition theorem. (b) For $q\in -\cC_{\pi-\ga'}$ and $\la\in \cC_\ga$,
$q\la\in -\cC_{\pi-\ga'+\ga}$, hence, $|1-\la q|$ is greater than or equal to the distance from the line $e^{i(\ga'-\ga)}\bR$
to 1. Hence, $|1-q\la|\ge \sin(\ga'-\ga)$, and it remains to apply \eq{norm_bound_main}. 
(c) follows from \eq{norm_bound_main}. 
(d) Rescaling, we may assume that $\|P\|=1$. 
Let $r=|q|<1$ and $\la\in [0,1]$. If $\Re q\le 0$, then $|(1-\la q)^{-1}|\le 1< 2|(1-q)|^{-1}$. 
If $\la\in [0,1/2]$, then $|(1-\la q)^{-1}|\le 2\le 4|1-q|^{-1}$. Finally, if
$\la\in (1/2,1]$ and $q=re^{i\varphi}$, $\varphi\in (-\pi/2,\pi/2)$, then 
\beqast
2|1-\la q|^2&=&2(1-2\la r\cos\varphi+\la^2r^2))=2(1-\la r)^2+4\la r(1-\cos\varphi)\\
&>& (1-r)^2+2r(1-\cos\varphi)=|1-q|^2.
\eqast

\end{proof}

\begin{rem}\label{rem:modif_cor}{\rm 
In applications to the evaluation of cpdfs of various kind, $\cH$ is the $L_2$-space with an appropriate exponential weight, equivalently,
$\cH=L_2$ after an appropriate Esscher transform. Next, typically, the transition operator $P$ is a bounded operator from $L_2$ to $C_b$, the space of continuous functions vanishing at infinity\footnote{or even to $\cS$, the space of 
infinitely differentiable functions vanishing at infinity faster than any polynomial, together with all derivatives}, and 
we are interested in bounds in the $C_b$-norm. To derive such a bound, we use
\[
(I-qP)^{-1}=I+qP(I-qP)^{-1},
\]
and the bounds $\|q(I-qP)^{-1}\|\le C|q|$ or $\|q(I-qP)^{-1}\|\le C|1-q/\|P\||^{-1}|q|$, and, assuming that $n>2$, modify the conclusions and proofs of Theorems \ref{thm: sep from 0}
and  \ref{thm:norm_inv_Z}  in the evident manner. Finally, we note that $P:L_2\to C_b$ (or $P:L_2 \to \cS$) is continuous.
}
\end{rem}

  \subsection{Evaluation of expectations  in symmetric random walk models}\label{s:Levy_symmetric}
   Let $Y, Y_j, j=1,2\ldots, $ be i.i.d. $\bR$-valued random variables on a probability space $(\Om,\cB, \bQ)$, and let $\bE$ be the expectation operator under $\bQ$. Let $\Phi$ be the characteristic function of $Y$ under $\bQ$.
For $x\in \bR$, $X_n=x+Y_1+\cdots + Y_n, n=0,1,2,\ldots,$ is a random walk on $\bR$ starting at $x$. In applications to finance, typically, $Y=Z_\barDe, \barDe>0,$ is an increment of a L\'evy process $Z$; the random walk appears implicitly when either a continuous time L\'evy model is approximated or options with discrete monitoring are priced. We impose the symmetry condition used in \cite{CarrLee09}
to justify the semi-static hedging of barrier options  (see \cite{Contrarian} for the discussion).
Under the same condition, representations of KoBoL (CGMY) and Meixner processes as a subordinated Brownian motion 
is derived in \cite{madan-yor}; in \cite{EfficientAmenable}, the representation result was extended to wide classes of Stieltjes-L\'evy processes (SL processes) and signed  SL processes (sSL processes). 
\vskip0.1cm
\noindent
{\sc Condition ($\Phi,\mathrm{sym}$).} 
{\em $\Phi$ admits analytic continuation  to a strip $S_{(\mum,\mup)}$, and there exists $\be\in (-\mup, -\mum)$ such that
  $\Phi(\xi-i\be)=\Phi(i\be-\xi)\ge 0,
\forall\ \xi\in \bR$. }
\vskip0.1cm
\noindent
Let  $h_-<h_+$, and 
let $\tau^-_{h_-}$ (resp., $\tau^+_{h_+}$) denote the first entrance time by $X$ into $(-\infty,\tau^-_{h_-}]$ (resp.,
$[h_+,+\infty)$). For a positive integer $n$ (maturity date), constant $q_0\in (0,1]$ (discount factor per time period), and a measurable function $G$ (payoff function at maturity), we consider expectations
\begin{enumerate}[(1)]
\item
$V_{eu}(G,q_0;n,x)=\bE^{\bQ,x}[q_0^n G(X_n)]$: the price of the European option; 
\item $V^-_{nt}(G, h_-, q_0;n,x)=\bE^{\bQ,x}[q_0^n G(X_n)\bfo_{\tau^-_{h_-}>n}]$ and
$V^+_{nt}(G, h_+, q_0;n,x)=\bE^{\bQ,x}[q_0^n G(X_n)\bfo_{\tau^+_{h_+}>n}]$: the prices of single barrier options with the lower and upper barriers $h_-$ and  $h_+$; 
\item
 $V_{nt}(G, h_-, h_+, q_0;n,x)=\bE^{\bQ,x}[q_0^n G(X_n)\bfo_{\tau^+_{h_+}\wedge \tau^-_{h_-}>n}]$: the price of the double barrier option which expires worthless if $X$ breaches one of the barriers before or on the maturity date.
\end{enumerate}
\vskip0.1cm
\noindent{\sc Condition ($G;\be$).} 
{\em
$G$ can be represented in the form $G(x)=e^{\be x}G_\be(x)$, where $G_\be\in L_2(\bR)$.
}
\begin{example}{\rm 
If $G$ is measurable and locally bounded, Condition ($G;\be$) is satisfied for
\begin{enumerate}[(a)]
\item
 double barrier options;
 \item
 down-and-out put options with the payoff $G(x)=(K-e^x)_+$;
 \item
 up-and-out call options with the payoff $G(x)=(e^x-K)_+$;
\item European put options and up-and-out put options;
\item
European call options and down-and-out put options.
\end{enumerate}
In cases (a)-(c), we can use any $\be\in \bR$ letting $G(x)=0, x\not\in [h_-,h_+]$, $G(x)=0, x\le h_-$, and $G(x)=0, x\ge h_+$, respectively.  In (d), $\be>0$, and in e), $\be<-1$.
}

\end{example}
Under the measure $\bQ_\be$ defined by $\frac{d\bQ_\be}{d\bQ}=e^{\be Y}/\Phi(-i\be)$,  the characteristic function
of $Y$ is $\bE^{\bQ_\be}[e^{i\xi Y}]=\bE^\bQ[e^{i(\xi-i\be) Y}/\Phi(-i\be)]=\Phi_\be(\xi):=\Phi(\xi-i\be)/\Phi(-i\be)$, and
we have 
\[
V_{eu}(G,q_0;n,x)=e^{\be x}\bE^{\bQ_\be,x}[ G_\be(X_n)]=(q_0\Phi(-i\be))^n e^{\be x}V_{eu;\be}(G_\be,1;n,x),
\]
where 
$ V_{eu;\be}(G_\be,1;n,x):=\bE^{\bQ_\be,x}[G_\be(X_n)] $ is the price of the European option in the model with the characteristic
function $\Phi_\be$, payoff $G_\be$, and discount factor 1.  In the same vein, the prices of types (2) and (3) are expressed via 
the prices in the same modified model, with the payoff function $G_\be$. 
Let $P_\be$ be the transition
operator in the modified model, and let $e_{(-\infty,h_+)}: L_2((-\infty,h_+))\to L_2(\bR)$, $e_{(h_-,+\infty)}: L_2((h_-,+\infty))\to L_2(\bR)$,
$e_{(h_-,h_+)}:L_2((h_-,h_+))\to L_2(\bR)$ be the extension-by-zero operators. Let $p_{(-\infty,h_+)}$, $p_{(h_-,+\infty)}$,
$p_{(h_-,h_+)}$ be the adjoints of $e_{(-\infty,h_+)}$, $e_{(h_-,+\infty)}$, $e_{(h_-,h_+)}$ (restriction operators on $(-\infty,h_+)$, $(h_-,+\infty)$, $(h_-,h_+)$, respectively).

We use the following evident fact.
\begin{lem}\label{lem:sym} Let the distribution of a random variable $Y$  on $\bR$  be infinitely divisible, and the characteristic function  satisfy $\Phi(\xi)=\Phi(-\xi)\ge 0, \xi\in\bR$. Then $0< \Psi(\xi)\le 1, \xi\in \bR$, and the expectation operator $P, Pu(x)=\bE[u(x+Y)]$, is a self-adjoint operator in $L_2(\bR)$,
with $\|P\|=1$.
\end{lem}
Applying Lemma \ref{lem:sym} and Theorem \ref{thm:norm_inv_Z}, we obtain the following proposition.
\begin{prop}\label{prop: justification_Z_oper} Let the distribution of $Y$ be infinitely divisible, and Conditions $(\Phi,\mathrm{sym})$ and $(G;\be)$ hold. Then
\begin{enumerate}[(a)]
\item
the following operators are non-negative self-adjoint operators, and their norms equal 1:
 $P_\be: L_2(\bR)\to L_2(\bR)$; $P^+_{h_+;\be}:=p_{(-\infty,h_+)}P_\be e_{(-\infty,h_+)}: 
L_2((-\infty,h_+))\to L_2((-\infty,h_+))$, $P^-_{h_-;\be}:=p_{(h_-,+\infty)}P_\be e_{(h_-,+\infty)}: 
L_2((h_-,+\infty))\to L_2((h_-,+\infty))$, \\ $P_{h_-,h_+;\be}:=p_{(h_-,h_+)}P_\be e_{(h_-,h_+)}: L_2((h_-,h_+))\to L_2((h_-,h_+))$;
\item
$V_{eu;\be}(G_\be,1;n,x)=(P^n_\be G_\be)(x)$, $V^\pm_{nt}(G, h_\pm, q_0;n,x)=((P^\pm_{h_\pm;\be})^n G_\be)(x)$,\\
$V_{nt}(G, h_-,h_+, q_0;n,x)=((P_{h_-,h_+;\be})^n G_\be)(x)$;
\item
in symmetric L\'evy models,  the application of the efficient $Z$-transform is justified for the evaluation European options and
single barrier and double options with discrete
monitoring (see \cite{EfficientDiscExtremum} and \cite{EfficientDiscDoubleBarrier}, respectively.)\end{enumerate}
\end{prop}
\vskip0.1cm
\noindent
{\sc Parameter choice.} We take a small $\ga>0$, and  use the prescription in 
Section \ref{param_choice_Z_SINH}. The constant in the bound for $H(f_n,d_n)$ is derived using 
the bounds in Theorem \ref{thm:norm_inv_Z}.

\section{European options I}\label{s:Euro_symm}

\subsection{Symmetric case}\label{ss:Euro_symm} 
Denote by $\hG$ the Fourier transform of $G$.
\begin{thm}\label{thm:Euro_symm}
Let $n>2$ and Conditions $(\Phi,\mathrm{sym})$ and $(G;\be)$ hold. 

Then, for any $r_0\in (0,1/\Phi(-i\be))$ and $\om_\ell\in (-\pi/2,0)$, there exist $\sg_\ell\in \bR$ and $b_\ell>0$
such that $\sg_\ell-b_\ell\sin(\om_\ell)=r_0$, and \bbe\label{eq:Euro_price_symm}
V_{eu}(G,q_0;n,x)
=\frac{q_0^n}{2\pi i}\int_{\cL_{L; \sg_\ell,b_\ell,\om_\ell}}dq\, q^{-n}\frac{1}{2\pi}\int_{\Im\xi=-\be}
e^{ix\xi}\frac{\Phi(\xi)\hG(\xi)}{1-q\Phi(\xi)}d\xi.
\ee
\end{thm}
\begin{proof} 
For any $r\in (0,1)$, 
\[
V_{eu,\be}(G,1;n,x)=\frac{1}{2\pi i}\int_{|q|=r}q^{-n-1}((I-qP_\be)^{-1}G_\be)(x)
=\frac{1}{2\pi i}\int_{|q|=r}q^{-n}(P_\be(I-qP_\be)^{-1}G_\be)(x).
\]
Take  $\om_\ell\in (-\pi/2, 0)$. Applying Theorems \ref{thm: sep from 0}-\ref{thm:norm_inv_Z} modified as in Remark \ref{rem:modif_cor}, we can choose $\sg_{\ell,\be}\in \bR$ and $b_{\ell,\be}>0$ so that $\sg_{\ell,\be}-b_{\ell,\be}\sin(\om_\ell)=r$ and  the contour $\{|q|=r\}$ can be deformed into $\cL_{L; \sg_{\ell,\be},b_{\ell,\be},\om_\ell}$. The result is
 \bbe\label{Veu3}
V_{eu,\be}(G,1;n,x)=\frac{1}{2\pi i}\int_{\cL_{L; \sg_{\ell,\be},b_{\ell,\be},\om_\ell}}dq\,q^{-n}(P_\be(I-qP_\be)^{-1}G_\be)(x).
\ee
Multiplying the both sides by $e^{\be x}(q_0\Phi(-i\be))^n$ and taking into account that $G_\be(x)=e^{-\be x}G(x)$ and
$e^{\be x}P_\be e^{-\be x}=\Phi(-i\be)^{-1}P$, we obtain 
\beqast
 V_{eu}(G,q_0;n,x)
&=&\frac{q_0^n\Phi(-i\be)^{n}}{2\pi i}\int_{\cL_{L; \sg_{\ell,\be},b_{\ell,\be},\om_\ell}}dq\, 
q^{-n}\Phi(-i\be)^{-1}P(I-q \Phi(-i\be)^{-1}P)^{-1}G)(x).
\eqast
It remains to change the variable $q_\be= q\Phi(-i\be)$,
 take into account that $P$ is the pdo\footnote{Recall that $A=a(D)$ is a pseudo-differential operator (pdo) with the symbol $a$
if $Au(x)=\cF^{-1}_{\xi\to x}a(\xi)\cF_{x\to \xi}u(x)$, where $\cF$ is the Fourier transform. See., e.g., \cite{eskin,NG-MBS}.}  with the symbol $\Phi(\xi)$, and
apply the Fourier transform $\cF$ to evaluate $(P(I-q_\be P)^{-1}G)(x)$. 
\end{proof}

\begin{rem}\label{rem: Euro_symm}
{\rm
\begin{enumerate}[(a)]
\item
Approximate optimal parameters of the sinh-deformation can be chosen using the algorithm in Section \ref{param_choice_Z_SINH}.

\item
Eq. \eq{eq:Euro_price_symm} is valid in the sense of  generalized functions. In order that \eq{eq:Euro_price_symm}
holds in the classical sense, we require that the product $\Phi\hG\in L_1(-i\be+\bR)$.
If $\Phi\hG\not\in L_1(-i\be+\bR)$, then the RHS of \eq{eq:Euro_price_symm} can be regularized
using the integration  by parts. See, e.g. \cite{NG-MBS}.
\end{enumerate}
}
\end{rem}

\subsection{Weak perturbation of symmetric distributions}\label{ss:Euro_non_symm}

\begin{prop}\label{prop:Euro_symm_small_pert} Let $n>2$, and Condition $(G;\be)$  and the following conditions hold:
  \begin{enumerate}[(i)]
  \item
  $\Phi=a_0\Phi_0+a_1\Phi_1$, where $a_0,a_1>0$, $a_0+a_1=1$, and $\Phi_0, \Phi_1$ are the characteristic functions
  of random variables $Y_0,Y_1$ on $\bR$;
  \item
   probability distribution of $Y_0$ is infinitely divisible;
  \item
  there exist $\mum<0<\mup$   such that  $\Phi_1$ and $\Phi_0$ admit analytic continuation 
  to $S_{(\mum,\mup)}$;
  
  \item
  there exists $\ga\in (0,\pi/2)$ such that, for 
  $\xi\in S_{(\mum,\mup)}$, 
   $\Phi(\xi)/\Phi_0(\xi)\in \cC_\ga$;
  \item
  there exists $\be\in (-\mup,-\mum)$ such that $\Phi_0(\xi-i\be)=\Phi_0(i\be-\xi)\ge 0$, $\xi\in \bR$.
  \end{enumerate}
Then  the statements of Theorem \ref{thm:Euro_symm} are valid.
\end{prop}
\begin{proof} Theorem \ref{thm: sep from 0} is applicable to $\Phi_\be$ because
 $\cF P_\be \cF^{-1}=\Phi_\be$ is the multiplication-by-$\Phi_\be(\xi)$ operator acting in the spaces of (generalized) functions on
the dual space, $\cF:L_2(\bR)\to L_2(\bR)$ is an isometry, and under conditions (i)-(v), 
$\Phi_\be(\xi)\in (\cC_\ga)^c$ for all $\xi\in \bR$.
\end{proof}

\subsection{Efficient realizations}\label{ss:eff_real} For $\gam< 0 <\gap$, define $\cC_{\gam,\gap}=\{\rho e^{i\varphi}\ |\ \rho>0, \gam<\varphi<\gap \ \mathrm{or}\ 
\pi-\gam<\varphi<\pi-\gap\}$.
 The characteristic exponents, hence, characteristic functions of 
increments $Y$ of wide classes of L\'evy processes admit analytic continuation to regions of the form $i(\mu,\mup)+(\cC_{\gam,\gap}\cup\{0\})$,
where  $\mum< 0< \mup, \mum<\mup, \gam< 0 <\gap$, and decay polynomially or faster as $\xi\to\infty$ in
$i(\mu,\mup)+(\cC_{\gam,\gap}\cup\{0\})$. The distributions satisfying this conditions are called {\em sinh-regular}. See 
 \cite{SINHregular,EfficientAmenable}.
The reader can construct examples of non-infinitely divisible distributions satisfying this condition as well.

On $G$, we impose the following additional condition.
\vskip0.1cm
\noindent
{\sc Condition ($\hG;a$).} $\hG(\xi)=e^{-ia\xi}\hG_0(\xi)$, where $a\in \bR$, and $\hG_0$ is a function analytic in $\bC\setminus i\bR$, and meromorphic in the complex plane with two cuts $i(-\infty, \mum]$ and $i[\mup,+\infty)$. Furthermore, for any $\ga\in (0,\pi/2)$, there exists
$C(\ga)>0$ such that 
\bbe\label{eq:boundhG0}
|\hG_0(\xi)|\le C(\ga)|\xi|^{-1},\ \xi\in \cC_\ga\cup (-\cC_\ga).
\ee 
\begin{rem}\label{rem:hG}{\rm
(a) Condition ($\hG;a$) is satisfied for the payoff functions of put and call options and digitals.
If $G$ is approximated by piece-wise polynomials, then $\hG$ is a sum of functions satisfying Condition ($\hG;a$).
This fact was used in \cite{levendorskii-xie-Asian,MarcoDiscBarr} where the conformal deformations technique was applied to price Asian and barrier options, respectively, with discrete monitoring. 

(b) The constructions below can be adjusted to the case of meromorphic functions having infinite number of poles outside $i\bR$,
in the same vein as in \cite{BSINH}, where the conformal deformations technique is used to calculate the coefficients in the B-spline method. In the pricing formulas with the deformed contours in the $\xi$-space, an additional sum of residues at the poles crossed
in the process of deformation appears.

(c) Under Conditions ($G;\be$) and ($\hG;a$), there exist $\mumpr<-\be<\muppr$ such that $\hG$ is analytic in $S_{(\mumpr,\muppr)}$.

(d) Condition ($\hG;a$) can be relaxed as follows. One may assume that $\hG$ admits analytic continuation to the union $\cU_{\hG}$ of a strip and cone around
the line $\{\Im\xi=-\be\}$. The deformed contours must be subsets of the intersection of $i(\mum,\mup)+(\cC_{\gam,\gap}\cup\{0\})$ and $\cU_{\hG}$.

}
\end{rem}

Under Condition ($\hG;a$), it is possible to decrease the CPU cost of a numerical realization of
the double integral on the RHS of \eq{eq:Euro_price_symm} as follows \cite{EfficientDiscExtremum}. 
Since $\Phi(\xi)>0$ for $\Im\xi=-i\be$, we can justify the deformation of the outer contour and choose the parameters 
of the deformation as described for symmetric operators in Section \ref{param_choice_Z_SINH};
the bound for the Hardy norm necessary to choose the step $\ze_\ell$
is derived using Theorem \ref{thm:norm_inv_Z}.
Then we make the change of variables
$q=\chi_{L;\sg_\ell, b_\ell, \om_\ell}(y)$ in the outer integral
\beqast
V_n&:=&\frac{1}{2\pi i}\int_{\cL_{L; \sg_\ell,b_\ell,\om_\ell}}dq\, q^{-n}\frac{1}{2\pi}\int_{\Im\xi=-\be}
e^{ix\xi}\frac{\Phi(\xi)\hG(\xi)}{1-q\Phi(\xi)}d\xi\\
&=&\frac{b_\ell}{2\pi }\int_{\bR}(\chi_{L; \sg_\ell,b_\ell,\om_\ell}(y))^{-n}\cosh(i\om_\ell+y)
 \frac{1}{2\pi}\int_{\Im\xi=-\be}
e^{i(x-a)\xi}\frac{\Phi(\xi)\hG_0(\xi)}{1-\chi_{L; \sg_\ell,b_\ell,\om_\ell}(y)\Phi(\xi)}d\xi,
\eqast
and apply  the simplified trapezoid rule to the outer integral:
\bbe\label{eq:VnEuroSimpTrap}
V_n=\sum_{|j|\le N_\ell}\ze_\ell f_{n,j}I(n,q_j),
\ee
where $q_j=\chi_{L; \sg_\ell,b_\ell,\om_\ell}(j\ze_\ell)$ and
\beqa\label{fnj}
f_{n,j}&=&\frac{b_\ell\ze_\ell}{2\pi }(\chi_{L; \sg_\ell,b_\ell,\om_\ell}(j\ze_\ell))^{-n}\cosh(i\om_\ell+j\ze_\ell),
\\\label{Inj}
I(q)&=&\frac{1}{2\pi}\int_{\Im\xi=-\be}
e^{i(x-a)\xi}\frac{\Phi(\xi)\hG_0(\xi)}{1-q\Phi(\xi)}d\xi.
\eqa
For each $|j|\le N_\ell$, we evaluate $I(q_j)$
deforming the line of integration $\{\Im\xi=-\be\}$ into a contour of the form
$\cL_{\om_1,b,\om}=\chi_{\om_1,b,\om}(\bR)$, where $\chi_{\om_1,b,\om}(y)=i\om_1+b\sinh(i\om+y)$, $\om_1\in \bR, b>0, \om\in (-\pi/2,\pi/2)$:
\bbe\label{Inj2}
I(q)=\frac{1}{2\pi}\int_{\cL_{\om_1,b,\om}}
e^{i(x-a)\xi}\frac{\Phi(\xi)\hG_0(\xi)}{1-q\Phi(\xi)}d\xi.
\ee
\begin{rem}\label{rem:sinh_par_choice}{\rm 
\begin{enumerate}[(1)]
\item
The parameters $\om_1,b,\om$ must be chosen so that, in the process of deformation, the deformed curve remains in the domain of analyticity of $\Phi$ and $\hG_0$, and
$1-q_j\Phi(\xi)\not\in (-\infty,0]$ for all $\xi$ on the curve; the choice of $\om_1,b,\om$ may depend on $q_j$ but, typically, it is possible to construct one deformation 
for all $q_j$ in \eq{eq:VnEuroSimpTrap}. 
\item
The sign of $\om$ depends on the sign of $x-a$ because it is advantageous to turn the oscillating factor into a decaying one. However, if there exists $\ka\ge 1$ such that $\Phi(\xi)=O(|\xi|^{-\ka})$ as $\xi\to\infty$ in a cone of analyticity of $\hG_0$ and $\Phi(\xi)$, then it is admissible (albeit non-optimal) to use $\om$ of either sign or $\om=0$.
 For details, see \cite{EfficientDiscExtremum}.
\end{enumerate}
}
\end{rem}

 \section{European options II}\label{s:Euro_Levy_non-symmetric}
  In this section, we consider pricing of European options in non-symmetric L\'evy models. 

 \subsection{Assumptions}
 Let $\psi(\xi)=-\ln\Phi(\xi)$; the branch of the logarithm is fixed by $\ln 1=0$. 
 \vskip0.1cm
 \noindent
 {\sc Condition $(\psi,\infty)$.} The characteristic exponent $\psi$ enjoys the following properties:
 \begin{enumerate}[(1)]
 \item
 there exist $\mum\le 0\le \mup$,
 $\mum<\mup$, such that $\psi$ admits analytic continuation to $\bC\setminus i((-\infty,\mum]\cup[\mup,+\infty))$;
 \item
 there exist $\mu\in \bR$ and $\nu_0>\nu_1>\cdots >\nu_N$ such that for any $\ga\in(0,\pi/2)$, 
 \bbe\label{eq:as_exp_psi}
\psi(\xi)+i\mu\xi=\sum_{j=0}^{N-1}d_j\xi^{\nu_j}+O(|\xi|^{-\nu_N}), \ (\cC_\ga\ni)\xi\to \infty;
\ee
\item
$2\ge \nu_0>0$;  $\nu_N<0$; $\Re d_j>0, j\le N-1$; $d_0>0$;  if $\nu_j= 1$; $d_j>0$;  if $\nu_0<1$, $\mu=0$.
\end{enumerate}

Thus, if the L\'evy process with the characteristic exponent $\psi$ is of finite variation, it must be a driftless process.
\footnote{
Using the identity
$\psi(-\bar\xi)=\overline{\psi(\xi)}$, one can write the analog of \eq{eq:as_exp_psi} for $\xi$ in the left half-plane.}

Condition \eq{eq:as_exp_psi} holds for the characteristic exponents of regular Stieltjes-L\'evy processes 
and signed Stieltjes-L\'evy-processes (regular sSL-processes and SL-processes), under additional conditions on the parameters characterizing the
process. The classes of sSL- and SL-processes are introduced in \cite{EfficientAmenable},
where it is demonstrated that essentially all popular classes of L\'evy processes and their mixtures are regular SL-processes
 satisfying \eq{eq:as_exp_psi}, with some exceptions when the log-terms or the terms with log-factors appear.
 For the sake of brevity, we do not consider a more general case when the terms of the form $\xi^\nu\ln \xi$ are present.
The constructions and results below admit straightforward  modifications to the more general case.

The characteristic exponent of an SL process  on $\bR$ is of the form
\bbe\label{eq:sSLrepr}
\psi(\xi)=(a^+_2\xi^2-ia^+_1\xi)ST(\cG^0_+)(-i\xi)+(a^-_2\xi^2+ia^-_1\xi)ST(\cG^0_-)(i\xi)+(\sg^2/2)\xi^2-i\mu\xi, 
\ee
where $\mu\in \bR$, $\sg\ge 0$, $\mum\le 0\le \mup$, $a^\pm_j\ge 0, j=1,2$, $\cG^0_\pm$ are Stieltjes measures supported on $[\pm\mu_\pm,+\infty)$,
and 
\[
ST(\cG)(z)=\int_{(0,+\infty)} (t+z)^{-1}\cG(dt)
\]
is the Stieltjes transform of the Stieltjes measure $\cG$. 
A function $\psi$ of the form \eq{eq:sSLrepr} admits analytic continuation to $\bC\setminus(i(-\infty, \mum]\cup i[\mup,+\infty))$. It is proved in \cite{EfficientAmenable} that if $\cG^0_\pm(dt)=p^\pm(t)dt$, \eq{eq:as_exp_psi} can be derived from asymptotic expansions of $p^\pm(t)$ as $t\to+\infty$.

We list several examples of regular SL-processes and distributions.

\begin{example}\label{ex: Phi_asymp_Phi_KBL}{\rm 
 A generic process of Koponen's family  \cite{genBS,KoBoL} is constructed as a mixture of spectrally negative and positive pure jump processes, with the L\'evy measure
\begin{equation}\label{KBLmeqdifnu}
F(dx)=c_+e^{\lm x}x^{-\nu_+-1}\bfo_{(0,+\infty)}(x)dx+
 c_-e^{\lp x}|x|^{-\nu_--1}\bfo_{(-\infty,0)}(x)dx,
\end{equation}
where $c_\pm>0, \nu_\pm\in [0,2), \lm<0<\lp$.
If $\nu_\pm\in (0,2), \nu_\pm\neq 1$, then the characteristic exponent is of the form $\psi(\xi)=-i\mu\xi+\psi^0(\xi)$,
where
\bbe\label{KBLnupnumneq01}
\psi^0(\xi)=c_+\Ga(-\nu_+)((-\lm)^{\nu_+}-(-\lm-i\xi)^{\nu_+})+c_-\Ga(-\nu_-)(\lp^{\nu_-}-(\lp+i\xi)^{\nu_-}).
\ee
Formulas in the case $\nu_\pm=0,1$ are different, and the log-factor appear. See \cite{genBS,KoBoL,NG-MBS,EfficientAmenable}.
The asymptotic expansion \eq{eq:as_exp_psi} is valid but the condition $d_0>0$ is satisfied only if $\nu_+=\nu_-$ and
$c_+=c_-$. Note that these two conditions were imposed in the series of examples in \cite{genBS,NG-MBS};
later, the authors of \cite{CGMY} renamed this subclass CGMY-model, and the labels for the parameters of the model were changed.

 In KoBoL model with $\nu_\pm =1$ and $c_+=c_-$, condition \eq{eq:as_exp_psi} with $d_0>0$ holds as well  \cite{NG-MBS,EfficientAmenable}.
}
\end{example}

\begin{example}\label{ex: Phi_asymp_Phi_NTS}{\rm
Condition \eq{eq:as_exp_psi} with $d_0>0$ holds for Normal inverse Gaussian (NIG) processes and the generalization: Normal Tempered Stable (NTS) processes, constructed
 in \cite{B-N,B-N-L}, respectively. The characteristic exponent is given by
 \begin{equation}\label{NTS2}
\psi^0(\xi)=\de[(\al^2-(\be+i\xi)^2)^{\nu/2}-(\al^2-\be^2)^{\nu/2}],
\end{equation}
where $\nu\in (0,2)$, $\de>0$, $|\be|<\al$; NIG obtains with $\nu=1$.
}\end{example}

\begin{example}\label{ex: Phi_asymp_Phi_mix_1}{\rm ($Y$ is an increment of a mixture of L\'evy processes). Let $\psi^k, k=1,2,\ldots, N,$ 
admit the asymptotic expansion
\eq{eq:as_exp_psi} with $\nu^1_0\ge \nu^2_0\ge \ldots \ge \nu^N_0$, and let $\psi^1$ satisfy 
 Condition ($\psi,\infty$). Then, 
under additional evident conditions  on  $d^k_j, k=2,\ldots, N, j\ge 0$ (e.g., all $\psi^k$ satisfy 
 Condition ($\psi,\infty$)), $\psi=\sum_{j=1}^N\psi^k$ satisfies Condition ($\psi,\infty$).

}
\end{example} 
\begin{example}\label{ex: Phi_asymp_Phi_mix_2}{\rm (Cpdf of $Y$ is a weighted sum of cpdfs).
Let the following conditions hold: 
\begin{enumerate}[(1)]
\item
 $\Phi=a_0\Phi^0+a_1\Phi^1$, where $a_0, a_1>0$, $a_0+a_1=1$, and $\Phi_k, k=0,1$,
 are the characteristic functions of a random variables on $\bR$;
 \item
 $\Phi^k, k=0,1$, admit analytic continuation to $\bC\setminus i\bR$;
 \item $\psi^0(\xi)=-\ln \Phi^0(\xi)$ is well-defined on $\bC\setminus i\bR$ and admits the asymptotic expansion
\eq{eq:as_exp_psi};
\item
$\Phi^1(\xi)/\Phi^0(\xi)$ is well-defined on $\bC\setminus i\bR$, assumes values in $\bC\setminus (-\infty,0]$, and decays faster than $|\xi|^{-\ka}$ 
as $\xi\to \infty$, where $\ka>0$.
\end{enumerate}
Then $\psi=-\ln \Phi$ admits the asymptotic expansion \eq{eq:as_exp_psi}.
}
\end{example}

\begin{example}\label{ex: Phi_asymp_Phi_VGP }{\rm In the popular Variance Gamma (VG) model introduced to Finance in \cite{MM91}, 
\begin{equation}\label{VGPexp}
\psi^0(\xi)=c[\ln(\al^2-(\be+i\xi)^2)-\ln (\al^2-\be^2)],
\end{equation}
where $\al>|\be|\ge 0$, $c>0$. The leading term in the analog of the asymptotic expansion \eq{eq:as_exp_psi}
is $c\ln\xi$. To consider the VG model, the constructions
in the paper need certain modifications.
}
\end{example}
Condition $\mu=0$ if $\nu_0<1$ excludes finite variation processes with non-zero drift and essentially asymmetric KoBoL 
processes ($\nu_+\neq \nu_-$ or $c_+\neq c_-$) but  allows for $\lp\neq -\lm$ ({\em essentially symmetric KoBoL}). The condition $d_0>0$ allows for mixtures of an essentially symmetric KoBoL
and asymmetric KoBoLs of lower orders.

\vskip0.1cm
\noindent
{\sc Notation.} 
Denote by $j_0$ the smallest $j>0$ such that 
 that $d_j>0$ for all $j< j_0$ and $\Im d_{j_0}\neq 0$. If such $j_0$ does not exist, set $j_0=N$.
 Note that $j_0\neq 1$, and
 define
 \[
 \barnu=\begin{cases}
 0, & \mu=0\ \text{and}\ \nu_{j_0}\le 0\\
 1, & \mu\neq 0\ \mathrm{and}\ \nu_{j_0}< 1\\
 \nu_{j_0}, & \nu_{j_0}> 1 \ \mathrm{or}\  \nu_{j_0}\in  (0,1) \  \mathrm{and}\ \mu=0.
 \end{cases}
 \]
 We identify $\bC_\xi$ with $\bR_x\times\bR_y$, and 
introduce the vector field 
 $\bV=(\dd_y\Im\psi(\xi), -\dd_x\Im\psi(\xi))$ on $\bC_\xi\setminus i\bR$.
\begin{defin}  $\cT$ 
denotes the set of trajectories $\cL$ of $\bV$ such that  $\Im\psi(\eta)=\de\neq 0, \eta \in \cL$.
If, in addition,
  $\cL$ is smooth, $\cL \subset \{\Re\xi>0\}$,  and $\cL\to \infty$, we write $\cL\in \cT_0$.
  \end{defin}
 Since $d_0>0$, it follows from  
    \eq{eq:as_exp_psi}, that if $\cL\in \cT_0$ and $(\cL\ni) x+iy\to \infty$, then $x\to+\infty$.

 \vskip0.1cm
 \noindent
 {\sc Condition $(\bV,\infty)$.} For any $-\be\in (\mum,\mup)$, $\eps>0$,  and $\de_0>0$, there exist infinitely many pairs $(\xi,\de)$
 in $U(\eps,\de_0)=\{\Re\xi>0, |\xi+i\be|<\eps, 0<|\de|<\de_0\}$
   such that $\cL$ starting at $\xi$ belongs to $\cT_0$ and satisfies $\Im\psi(\eta)=\de, \eta\in \cL$.
  
 We verify Condition $(\bV,\infty)$ (in fact, a stronger property) for SL-processes. 

 \begin{lem}\label{lem:vect field} Let $X$ be an SL-process, and let \eq{eq:as_exp_psi} hold. 
Then for almost all $\de$, any $\cL\in \cT$ satisfying $\Im\psi(\xi)=\de, \xi\in \cL$, 
belongs to $\cT_0$.
\end{lem}
 The proof is in Section \ref{proof_lem:vect field}.

  \subsection{Main Theorem}
   In Section \ref{s:Euro_symm}, the outer contour of integration $\{|q|=r\}$ is deformed (see Theorem \ref{thm:Euro_symm}) keeping the same
   line of integration in the inner contour (integration w.r.t. $\xi$).
The inner contour  is deformed later, to improve the efficiency of the numerical algorithm for
the evaluation of the individual terms in the simplified trapezoid rule (see Section \ref{ss:eff_real}). In the asymmetric case, we can justify the deformation of the contour in the $Z$-inversion formula only after the line of integration w.r.t. $\xi$ is deformed into a contour $\cL$ of a special form. 
After that, we make the corresponding sinh-change of variables, apply the simplified trapezoid rule, and calculate each term making the sinh-change of variables in the integrals w.r.t. $\xi$ as in the symmetric case.

For $d>0$, $c,C>0$, $\al\in (0,1)$ and $\be$ in Condition ($G;\be$), denote $\cL_{up}(d,C,\al)=\{\xi\in \bC\ |\ \Im\xi=-\be+d+C|\Re\xi|^\al\}$
and  $\cL_{do}(d,C,\al)=\{\xi\ |\ \Im\xi=-\be-d-C|\Re\xi|^\al\}$.
  \begin{thm}\label{thm:Euro-non-symm0} Let Conditions  $(\psi,\infty)$,  $(\bV,\infty)$, $(G;\be)$, $(\hG;a)$ hold, and either $\nu_0>(\barnu+1)/2$ or $\nu_0\in (0,1)$ and $\barnu=0$. Then, for any $\ga\in (0,\pi/2)$ and $r\in (0,1/\Phi(-i\be))$, there exist $d>0$, $c,C>0$, 
  $\al\in (0,1)$, $\ka>0$ and a smooth curve $\cL$ sandwiched between  $\cL_{up}(d,C,\al)$ and $\cL_{do}(d,C,\al)$ such that 
\begin{enumerate}[(a)]
\item
$\cL$ admits a parametrization $\bR\ni y\mapsto \chi_\cL(y)\in \cL$ such that
\begin{enumerate}[(i)]
\item
$\chi_\cL$ admits analytic continuation to a strip $S_{(-d_\cL,d_\cL)}$, where $d_\cL\in (0, d)$;
\item
for any $\xi\in\chi_\cL(S_{(-d_\cL,d_\cL)})$, $\Phi(\xi)\in \cC_{\ga/2}$;
\item
as $x\to\pm\infty$,  $\chi_\cL(x+ia)=x+O(|x|^\al)$ and
$\dd_x\chi_\cL(x+ia)=1+O(|x|^{-\ka})$, uniformly in $a\in (-d_\cL,d_\cL)$;
\end{enumerate}
\item
for $(q,\xi)\in ((-\cC_{\pi-\ga})\cup\cD(0,r))\times \chi(S_{(-d_\cL,d_\cL)})$, $1-q\Phi(\xi)\not\in (-\infty,0]$, and $|1-q\Phi(\xi)|\ge c$; 
\item
\bbe\label{eq:Euro_non_symm_10}
V_{eu}(G,q_0;n,x)=\frac{q_0^n}{2\pi i}\int_{|q|=r}\frac{dq}{q^{n}}\frac{1}{2\pi}\int_{\cL}
e^{i(x-a)\xi}\frac{\Phi(\xi)\hG_0(\xi)}{1-q\Phi(\xi)}d\xi;
\ee
\item
if $n>1$, then, for any  $\om_\ell\in (\ga/2,\pi/4)$ and $r\in (0, 1/\Phi(-\be))$,  there exist $\sg_\ell\in \bR$ and $b_\ell>0$ such that
$\sg_\ell-b_\ell\sin \om_\ell=r$ and
\bbe\label{eq:Euro_non_symm_20}
V_{eu}(G,q_0;n,x)=\frac{q_0^n}{2\pi i}\int_{\cL_{L;\sg_\ell,b_\ell,\om_\ell}}\frac{dq}{q^{n}}\frac{1}{2\pi}\int_{\cL}
e^{i(x-a)\xi}\frac{\Phi(\xi)\hG_0(\xi)}{1-q\Phi(\xi)}d\xi.
\ee
\end{enumerate}
\end{thm}
The proof of Theorem \ref{thm:Euro-non-symm0} and an explicit construction of $\cL$ occupy the rest of the section. 

\begin{rem}{\rm To choose $\sg_\ell,b_\ell,\om_\ell$ for an efficient numerical evaluation, we take a small $\ga>0$, and  use the prescription in 
Section \ref{param_choice_Z_SINH}  to choose the parameters of the deformation of the contour $\{|q|=r\}$. 
To derive a bound for the Hardy norm, hence, choose $\ze_\ell$, we use the bounds $|\Phi(\xi)|\le Ce^{-(d_0/2)/|\xi|^{\nu_0}}$ and  $|1-q\Phi(\xi)|\ge c$
 to find $C$ such that  
the inner integral
\bbe\label{Inj_Def}
I(q):=\frac{1}{2\pi}\int_{\cL}
e^{i(x-a)\xi}\frac{\Phi(\xi)\hG_0(\xi)}{1-q\Phi(\xi)}d\xi.
\ee
admits a bound
$|I(q)|\le C$ for $q\in (-\cC_{\pi-\ga})\cup\cD(0,r)$.
The truncation parameter $\La=N_\ell \ze_\ell$ is chosen  as in Section \ref{param_choice_Z_SINH}, using the bound
for $|I(q)|$.
After that,  
we make the change of variables in the outer integral on the RHS of \eq{eq:Euro_non_symm_20},  and apply the simplified trapezoid rule w.r.t. $y$ as in Section \ref{s:Euro_symm}. Finally, for each $q_j=q(y_j)$ appearing in the simplified trapezoid rule,
we deform the contour $\cL(\de; u)$ into an appropriate contour of the form $\cL_{\om_1, b,\om}$. For the choice of
$\om_1, b,\om$, see Remark \ref{rem:sinh_par_choice}.

}\end{rem}

 \subsection{Deformations of the contour in the $\xi$-space
}\label{sss:constr_def}

First, we study the asymptotics of the trajectories at infinity.
By a slight abuse of notation, we write $\cL_\de$ to indicate the value of $\Im\psi$ along $\cL\in\cT_0$.
We do not claim that, given $\de\neq 0$, $\cL_\de$ is unique. The general properties formulated and proved below are valid for
any choice of $\cL_\de$. 

The following lemma is immediate from \eq{eq:as_exp_psi}.
\begin{lem}\label{param1} Let Conditions $(\psi,\infty)$ and $(\bV,\infty)$ hold. 
Then,
for any bounded interval $[\de_-,\de_+]\subset\bR$, there exist $R>0$ such that, for any triple $(\de, \cL_\de, x)\in 
[\de_-,\de_+]\times  \cT_0\times [R,+\infty)$,
 there exists a unique $y(\cL_\de; x)\in \bR$ 
such that $x+iy(\cL_\de;x)\in \cL_\de$.
\end{lem}
It is straightforward to derive from \eq{eq:as_exp_psi} asymptotic formulas for $y(\cL_\de; x)$, and these formulas serve as the
basis for the construction of deformed contours in the formula for the joint cpdf. 

 \begin{lem}\label{lem:asymp_psi_SL}
Let  $\psi$ satisfy Condition $(\psi,\infty)$. Then
\begin{enumerate}[(a)]
\item
there exists $\ka>0$ such that for any $\de\neq 0$
\bbe\label{eq:as_y_de}
y(\cL_\de;x)=p(\de)x^{\barnu+1-\nu_0}+O(x^{\barnu+1-\nu_0-\ka}),\ x\to+\infty,
\ee 
where $p(\de)$ is the same for all trajectories $\cL_\de$:
\bbe\label{eq:def_p_de}
p(\de)=\begin{cases}
 \de/(d_0\nu_0), & \barnu<0 \\
 (\de-\Im d_{j_0})/(d_0\nu_0), & \nu_{j_0}=0 \ \text{and}\ \mu=0\\
 -\Im d_{j_0}/(d_0\nu_0), & \nu_{j_0}\in (0,1) \ \text{and}\ \mu=0 \ \text{or}\  \nu_{j_0}\in (1,2)\\
 \mu/(d_0\nu_0), & \nu_{j_0}<1 \ \text{and}\ \mu\neq 0;
 \end{cases}
\ee
\item
for any $a<b$ there exist $c>0$ and $R>0$ such that if $a<\de_-<\de_+<b$, $\de_j\neq 0$, and $\de_+-\de_-<c$,
then
\bbe\label{eq:dif:y_de}
    y(\cL_{\de_+};x)-y(\cL_{\de_-};x)\ge \frac{\de_+-\de_-}{2d_0\nu_0}x^{1-\nu_0}, \ x\ge R;
\ee
\item 
 \bbe\label{eq:asymp_der_y}
 \dd_x y(\cL_\de;x)=O(x^{-\min\{1,\nu_0-\barnu\}}), \ x\to +\infty.
 \ee
\end{enumerate}
\end{lem}

\begin{proof} Since $d_0>0$, it  follows from 
\eq{eq:as_exp_psi} that $y(\cL_\de;x)/x\to 0$ as $x\to+\infty$. Hence, we can use the Taylor expansion 
\[
(x+iy)^{\nu_j}=x^{\nu_j}+i\nu_j yx^{\nu_j-1}-\frac{\nu_j(\nu_j-1)}{2}y^2x^{\nu_j-2}-i\frac{\nu_j(\nu_j-1)(\nu_j-2)}{6}y^3x^{\nu_j-3}
+\cdots
\]
Let $d_j=d^r_j+id^i_j$, where $d^r_j, d^i_j\in\bR$. Then
\beqast
\Im d_j(x+iy)^{\nu_j}&=&d^r_j\left(\nu_j yx^{\nu_j-1}-\frac{\nu_j(\nu_j-1)(\nu_j-2)}{6}y^3x^{\nu_j-3}
+\cdots \right)\\
&&+d^i_j\left(x^{\nu_j}-\frac{\nu_j(\nu_j-1)}{2}y^2x^{\nu_j-2}+\cdots\right).
\eqast
For $y=y(\cL_\de;x)$, we obtain
\beqa\label{eq:as_Im}
\de&=&\sum_{j=0}^N\left[d^r_j\left(\nu_j yx^{\nu_j-1}-\frac{\nu_j(\nu_j-1)(\nu_j-2)}{6}y^3x^{\nu_j-3}
+\cdots \right)\right.\\\nonumber
&&\left.+d^i_j\left(x^{\nu_j}-\frac{\nu_j(\nu_j-1)}{2}y^2x^{\nu_j-2}+\cdots\right)\right]-\mu x+O(x^{-\nu_N}).
\eqa
Eq. \eq{eq:as_y_de}-\eq{eq:def_p_de} are immediate from \eq{eq:as_Im}. 
 To prove \eq{eq:dif:y_de}, we may replace $\Im\psi$ with $\Im\psi_{mod}$, the sum 
of a finite number of terms on the RHS of \eq{eq:as_Im}, without the $O$-term, and  $y(\cL_\de;x)$
with  the unique solution $y_{mod}(\de;x)$ of the modified equation $\de=\Im\psi_{mod}(x+iy)$.  We differentiate  $\Im\psi_{mod}$
and use \eq{eq:as_exp_psi}, \eq{eq:dif:y_de} and the inequality $\barnu<\nu_0$ to conclude that
there exists $\ka>0$ such that for any $\de\neq 0$
\bbe\label{eq:as_y_de_der}
\dd_y\Im\psi_{mod}(x+iy)\vert_{y=y_{mod}(\de;x)}=d_0\nu_0x^{1-\nu_0}(1+O(x^{-\ka})),\ x\to+\infty.
\ee 
Applying \eq{eq:as_y_de_der} and the implicit function theorem, we derive 
\bbe\label{eq:asymp_der_de}
\dd_\de y_{mod}(\de;x)=(1/(d_0\nu_0))x^{1-\nu_0}(1+O(x^{-\ka})),
\ee where $\ka>0$, and 
\eq{eq:dif:y_de} follows. The proof of \eq{eq:asymp_der_y} is analogous to the proof of \eq{eq:asymp_der_de}.
\end{proof}
Fix $-\be\in (\mum,\mup)$. For any $\de^*>0$, there exists $\eps>0$ such that for all $\xi\in B(-i\be,\eps):=\{\xi\ |\ |\xi-iy|<\eps, \Re\xi> 0\}$, $0<|\Im\psi(\xi)|<\de^*$.  Fix $\de\in (-\de^*,\de^*), \de\neq 0,$  and a trajectory $\cL_\de\in \cT_0$ starting at
$\xi\in B(-i\be,\eps)$. Fix $u\in (-\be-\eps, -\be+\eps)$ and extend $\cL$ to a smooth curve starting at $iu$ so that
  the added part is a subset of $B(-i\be,\eps)$, and
the extended curve is orthogonal to $i\bR$.
Finally, extend the curve further still to the curve symmetric w.r.t. $i\bR$. By a slight abuse of notation, 
denote the resulting curve $\cL(\de; u)$. Two curves $\cL(\de_-; u_-)$ and $\cL(\de_+; u_+)$
with $\de_-<\de_+$ are constructed so that $\cL(\de_-; u_-)\cap \cL(\de_+; u_+)=\emptyset$.
\vskip0.1cm
\noindent
{\sc Notation.} 
We write $\cL(\de;u)\subset \cT_{\mathrm{ext}}$ and $\cL(\de_-; u_-)\prec \cL(\de_+; u_+)$.

\begin{lem}\label{lem:extended curves_0} Let $\cL(\de;u), \cL(\de_\pm, u_\pm)\in \cT_{\mathrm{ext}}$. Then
\begin{enumerate}[(a)]
\item
for all $\xi\in \cL(\de;u)$, $\Im\psi(\xi)\in (-\de^*,\de^*)$,
and, as $x\to\pm\infty$ and $(x,y)\in \cL(\de,u)$,
\bbe\label{eq:as_y_de0}
y=y(\cL(\de,u),x):=p(\de)|x|^{\barnu+1-\nu_0}+O(|x|^{\barnu+1-\nu_0-\ka});
\ee
\item
if  $\cL(\de_-; u_-)\prec \cL(\de_+; u_+)$, then there exists $c>0$ such that,
for $(x,y_\pm)\in \cL(\de_\pm,u_\pm)$,
\bbe\label{eq:dist_cL_dem}
\mathrm{dist}((x,y_\mp),  \cL(\de_\pm,u_\pm))\ge c(1+|x|)^{1-\nu_0}.
\ee
\end{enumerate}
\end{lem}
\begin{proof} Both statements are immediate the construction of curves $\cL(\de; u)$,  equality $\overline{\psi(\xi)}=\psi(-\bar\xi)$ and \eq{eq:as_y_de}-\eq{eq:dif:y_de}.
\end{proof}
\begin{rem}\label{rem:param_cL}{\rm Lemma \ref{lem:extended curves_0} allows us to parametrize $\cL(\de,u)$ using a smooth
map $\bR\mapsto x(t)+iy(t)\in \cL(\de,u)$ such that $x(t)=t$ in  neighborhoods of $\pm \infty$, and $(x(-t), y(-t))=(-x(t), y(t))$. In the neighborhoods of $\pm\infty$, we
may and will use $x$ as a parameter.
}
\end{rem}  


\subsection{Proof of Theorem \ref{thm:Euro-non-symm0}}
Under Conditions ($G;\be$) and ($\hG;a$), $-i\be$ is not a pole of $\hG$. We fix $\de^*\in (0,\pi/2)$, $\de\in (-\de^*,\de^*), \de\neq 0$, $u\in \bR$ and a curve $\cL(\de; u)$ 
such that the interval connecting $iu$ and $-i\be$ is in the domain of analyticity of $\hG$ and $\Phi$, and $\Im\psi(\xi)\in (-\de^*,\de^*)$ for all
$\xi\in \cL(\de;u)$. 
\begin{thm}\label{thm:Euro-non-symm1} Let Conditions $(\psi,\infty)$,  $(\bV,\infty)$, $(G;\be)$, $(\hG;a)$ hold, $n>1$, and  $\nu_0>(\barnu+1)/2$. Then  there exist $\de^*$, $\de$, $u$, and  $\cL(\de;u)\in\cT_{\mathrm{ext}}$  such that 
Theorem \ref{thm:Euro-non-symm0} holds with $\cL=\cL(\de;u)$. 
\end{thm}
 \begin{proof} Let $r\in (0,1)$.  
 On the strength of Conditions $(\psi,\infty)$,  $(\bV,\infty)$, $(G;\be)$, $(\hG;a)$, and Lemma \ref{lem:extended curves_0}, there exists a simply connected region
 $\cU$  of analyticity of the inner integrand on the RHS' of \eq{eq:Euro_non_symm_10} amd \eq{eq:Euro_non_symm_20} such that (1) $\Phi(\xi)\in \cC_{\de^*}$ for all $\xi\in \cU$, 
 and, therefore, $|1-q\Phi(\xi)|\ge c>0$, where $c>0$ is independent of $q\in \cD(0,r)$ and $\xi\in \cU$,
 and (2)
 $\{\Im\xi=-i\be\}\subset \cU$,
  $\cL(\de;u)\subset\cU$. Hence, for each $q\in \cD(0,r)$, the contour $\{\Im\xi=-i\be\}$ can be deformed into $\cL(\de;u)$ remaining in $\cU$.
 Furthermore, in view of \eq{eq:as_y_de0}, in the process of the deformation, the factor $e^{i(x-a)\xi}$ is uniformly bounded if
 $(x-a)p(\de)>0$, and admits a bound via $e^{p(\de)|\Re\xi|^\ka}$, where $\ka=\barnu+1-\nu_0<\nu_0$.  Since $|\Phi(\xi)|\le Ce^{-d|\Re \xi|^{\nu_0}}$, for any $d<d_0$, the integrand decays faster than an exponential function, hence,
 the deformation is justified. This proves \eq{eq:Euro_non_symm_10}.
 
 Next, note that 
  $\Phi(\xi)\in \cC_{\de^*}$ for all $\xi\in \cU$, and therefore, for any $\ga\in (\de^*,\pi/2)$ and
$q\in (-\cC_{\pi-\ga})^c\cup\cD(0,r)$, $1-q\Phi(\xi)\not\in (-\infty,0]$ and bounded away from 0.
Hence, the same argument as in the proof of Lemma \ref{lem:Z-SINH} can be used to choose
$\sg_\ell$ and $b_\ell$ so that, in the process of deformation of the contour $\{|q|=r\}$ into $\cL_{L;\sg_\ell,b_\ell,\om_\ell}$, 
$1-q\Phi(\xi)\not\in (-\infty,0]$ and bounded away from 0,
uniformly in $q$ and $\xi$. Since $n>1$, the integrand is bounded by an absolutely integrable function,
hence, the transformation of the outer contour is justified, and \eq{eq:Euro_non_symm_20} is proved.
 
 \end{proof}

If $\nu_0\le (\barnu+1)/2$, then $\nu_0\in (0,1)$. In this case, assuming that $\barnu=0$, we can essentially repeat the proof of Theorem \ref{thm:Euro-non-symm1} 
using the following modification of the contour $\cL(\de;u)$. Since $\barnu=0$, all $d_j>0$ if $\nu_j>0$. Therefore,
$\barnu+1-\nu_0=1-\nu_0<1$, and, if  $\nu_j>0$ and $j>0$, then, for $C>0$ fixed,  
\bbe\label{eq:asymdj}
\Im d_j(x+iy(x))^{\nu_j}=O(x^{\nu_j}x^{1-\nu_0-1})=o(1), \ \mathrm{as}\ x\to+\infty,
\ee
uniformly in $y(x)\in (-Cx^{1-\nu_0}, Cx^{1-\nu_0})$. For $j=0$, we have
\bbe\label{eq:asymdj0}
\Im d_0(x+iy(x))^{\nu_0}=(d_0\nu_0)y(x) x^{\nu_0-1}+o(1), \ \mathrm{as}\ x\to+\infty.
\ee 
We fix $x^*>0$, define 
\[
y^*(x^*;\cL(\de,u),t)=\begin{cases}
y(\cL(\de,u); x^*), & t>x^* \\
y(\cL(\de,u); t), & -x^*\le t \le x^* \\
y(\cL(\de,u); -x^*), & t<-x^*,
\end{cases}
\]
and $\cL^*(x^*;\de,u)\vert_{|t|\ge x^*}=\{t+iy^*(x^*;\cL(\de,u),t)\ |\ |t|\ge x^*\}$ and 
$\cL^*(x^*;\de,u)\vert_{|t|\le x^*}=\cL(\de,u)\vert_{|t|\le x^*}$.
Since 1) $-\de^*<\de <\de^*$, 2) $\Im\psi(x+iy(\cL_\de,x))=\de$ for $|x|\ge x^*$, if $x^*$ is sufficiently large, 3)
$\Im\psi(\xi)\in (-\de^*,\de^*), \xi\in \cL(\de,u)$, and 4) \eq{eq:asymdj}
and \eq{eq:asymdj0} hold,  we conclude that if $x^*$ is sufficiently large, then  $\Im\psi(\xi)\in (-\de^*,\de^*), \xi\in \cL^*(x^*; \de,u)$.
Since $y^*(x^*;\de,\cdot)$ is uniformly bounded, we can repeat the proof of Theorem \ref{thm:Euro-non-symm1} (a), (b),
with evident changes, and obtain
\begin{thm}\label{thm:Euro-non-symm2} Let Conditions  $(\psi,\infty)$,  $(\bV,\infty)$, $(G;\be)$, $(\hG;a)$ hold, $n>1$,
$\nu_0\in (0,1)$,  $d_j>0$ if $\nu_j\ge 0$, and let $\de^*, \de, x^*, \cL^*(x^*;\de,u)$ be as described above. Then
the statements of Theorem \ref{thm:Euro-non-symm1} are valid with $\cL^*(x^*;\de,u)$ in place of $\cL(\de,u)$.
\end{thm}

\section{Pricing barrier options with discrete monitoring in L\'evy models}\label{s:barrier_Levy_non_symm}  
\subsection{General formulas}\label{ss:single_gen}
 For a measurable function $G$ and barrier $h$, consider 
the up-and-out option with the payoff $G(X_n)$ at maturity date $n$. We may assume that the discount factor $q_0= 1$.
Let  $\barX_n=\max_{0\le m\le n}X_m$ and $\uX_n=\min_{0\le m\le t}X_m$ be the maximum
and minimum processes (defined path-wise, a.s.); $X_0=\barX_0=\uX_0=0$. Then the option price at time 0 and spot $x$ can be written as
$
V(G, h;n;x)=\bE[G(x+X_n)\bfo_{x+\barX_n<h}].
$
At the first step, as in \cite{FusaiAbrahamsSguarra06}, where barrier options with discrete monitoring in the Brownian motion model  are priced, we make the discrete Laplace transform ($Z$-transform) of the series $\{V(G,h;n,x)\}_{n=0}^\infty$. 
Next, for $q\in (0,1)$, the discrete Laplace transform $\tV(G,h;q,x)$ is calculated using the expected present value operators (EPV-operators) technique which is the operator form of the Wiener-Hopf factorization technique. Let $T_q$ be a random variable with the distribution $\bE[T_q=n]=(1-q)q^n$, independent of $X$, and let $u$ be a bounded measurable function. Then
$\cEq u(x)=\bE[u(x+X_{T_q})]$,
  $\cEpq u(x)=\bE[u(x+\barX_{T_q})]$ and $\cEmq u(x)=\bE[u(x+\uX_{T_q})]$ define the action of the EPV-operators under the the random walk $X$ and minimum and maximum processes. The EPV operators $\cEq$ and $\cE^\pm_q$ are pseudo-differential operators (pdo) with the symbols $(1-q)/(1-q\Phi(\xi))$ and $\phi^\pm_q(\xi)$, respectively, where 
 $\phipq(\xi)=\bE[e^{i\barX_{T_q}\xi}]$ and $\phimq(\xi)=\bE[e^{i\uX_{T_q}\xi}]$ are the Wiener-Hopf factors. Recall the Wiener-Hopf
 identity
 \bbe\label{whf1}
 \phipq(\xi)\phimq(\xi)=\frac{1-q}{1-q\Phi(\xi)}, \ \xi\in \bR, 
 \ee
 and its operator form
 \bbe\label{whfoper}
 \cEpq\cEmq=\cEq.
 \ee
  Let $G$ be uniformly bounded and measurable. Then, as it is proved in \cite{perp-Bermuda,NG-MBS,IDUU,EfficientDiscExtremum} in increasing  generality,
  \bbe\label{eq:perp_barr}
 \tV(G,h;q,x) =(1-q)^{-1}(\cEpq\bfo_{(-\infty, h)}\cEmq G)(x).
 \ee
As it is stated in \cite[Remark 3.3]{EfficientDiscExtremum}, \eq{eq:perp_barr} (Eq. (3.14) in \cite{EfficientDiscExtremum}) is less
 efficient for a numerical realization than the following representation   \cite[Eq. (3.13)]{EfficientDiscExtremum}:
  \bbe\label{eq:perp_barr_2}
 \tV(G;h;q,x)=G(x)+(q\Phi(D)(1-q\Phi(D))^{-1}G)(x)-(1-q)^{-1}(\cEpq\bfo_{[h,+\infty)}\cEmq G)(x),
 \ee
 especially if $\hG$ does not decay at infinity sufficiently fast.
Assuming that Conditions ($\phi,\infty$), ($G;\be$) and ($\hG;a$) hold, there exist $r>0$ and a strip $S_{(\mumpr,\muppr)}$ around $\{\Im\xi=-\be\}$ such that (i) $\hG$ is analytic on $S_{(\mumpr,\muppr)}$; (ii)
$(1-q)/(1-q\Phi(\xi))$,  $\phi^\pm_q(\xi)$ and the reciprocals $1/\phi^\pm_q(\xi)$ 
are analytic w.r.t. $q$ in the disc $\cD(0,r)$, admit analytic continuation  w.r.t. $\xi$ to $S_{(\mumpr,\muppr)}$; (iii)
there exist $C>0$ and $\al$ such that the functions listed in (ii) are bounded (in absolute value) by $C(1+|\xi|)^\al$ for
all $q\in \cD(0,r)$ and $\xi\in S_{(\mumpr,\muppr)}$.

Let $n>1$. 
 We apply the inverse $Z$-transform and calculate $V(G,h;n,x)$, for 
 and $x<h$, 
 \bbe\label{eq:perp_barr_3}
 V(G;h;n,x)=\frac{1}{2\pi i}\int_{|q|=r}dq\, q^{-n}\int_{\Im\xi=-\be}e^{i(x-a)\xi}
 \frac{\Phi(\xi)\hG_0(\xi)}{1-q\Phi(\xi)}d\xi-V_1(G,h;n,x),
 \ee
 where $r\in (0,1/\Phi(-i\be))$ is sufficiently small, and
 \bbe\label{eq:V1Gh}
 V_1(G,h;n,x)=\frac{1}{2\pi i}\int_{|q|=r}dq\, q^{-n-1}(1-q)^{-1}(\cEpq\bfo_{[h,+\infty)}\cEmq G)(x).
 \ee
The first term on the RHS of \eq{eq:perp_barr_3} is the price of a European option; the conditions for
the efficient inverse $Z$-transform are derived in Section \ref{s:Euro_Levy_non-symmetric}. Since $\phimq(\xi)^{-1}$ is analytic in the half-plane 
$\{\Im\xi<\muppr\}$, $((\cEmq)^{-1}\bfo_{[h,+\infty)}\cEmq G)(x)=0, x<h$. Hence, using \eq{whfoper}, we may rewrite \eq{eq:V1Gh}
as 
 \bbe\label{eq:V1Gh2}
 V_1(G,h;n,x)=\frac{1}{2\pi i}\int_{|q|=r}dq\, q^{-n}\Phi(D)(1-q\Phi(D))^{-1}(\cEmq)^{-1}\bfo_{[h,+\infty)}\cEmq G)(x).
 \ee

\subsection{Integral representations of  $V_1(G,h;n,x)$}\label{ss:repr_V1} 
Explicit formulas for  $V_1(G,h;n,x)$ involve double or triple integrals; the integrands are expressed in terms of  the Wiener-Hopf factors. 
 The formulas for and  properties of the Wiener-Hopf factors are well-known, see, e.g., \cite{Borovkov1,NG-MBS,IDUU,EfficientDiscExtremum}. 
\begin{prop}\label{prop_WHF} ( \cite[Prop. 3.2]{EfficientDiscExtremum}) Let $\Phi$ admit analytic continuation to a strip $S_{(\mum,\mup)}$, where $\mum\le 0\le \mup$, and $\mum<\mup$.
Then, for any $q\in (0,1)$, 
\begin{enumerate}[(a)]
\item
there exist $\lm\ge \mum$ and $\lp\le \mup$ s.t. $\lm<\lp$, and $c>0$ such that 
\bbe\label{eq:boundPhiq}
\Re (1-q\Phi(\xi))\ge c, \quad \xi\in S_{(\lm,\lp)};
\ee
\item
furthermore, for any $\xi$ in the half-plane $\{\Im\xi>\lm\}$ and any $\omm\in (\lm, \Im\xi)$,
\beqa\label{phip1}
 \phipq(\xi)&=&\exp\left[-\frac{1}{2\pi i}\int_{\Im\eta=\omm}\frac{\xi\ln((1-q)/(1-q\Phi(\eta)))}{\eta(\xi-\eta)}d\eta\right],
 \eqa
 and 
 for any $\xi$ in the half-plane $\{\Im\xi<\lp\}$ and any $\omp\in (\Im\xi,\lp)$,
\beqa\label{phim1}
 \phimq(\xi)&=&\exp\left[\frac{1}{2\pi i}\int_{\Im\eta =\omp}\frac{\xi\ln((1-q)/(1-q\Phi(\eta)))}{\eta(\xi-\eta)}d\eta\right].
 \eqa
   \end{enumerate}
\end{prop}

The following proposition is a refined form of Proposition  \ref{prop_WHF}. The proof is a straightforward modification
of the proof of Proposition  \ref{prop_WHF}, in the same vein as the analog of Proposition  \ref{prop_WHF} for L\'evy processes in \cite{NG-MBS}
is refined in \cite{Contrarian}.
\begin{prop} \label{rem:wide_disc} Let Condition $(\psi,\infty)$ hold. Then, 
for any $[\lm,\lp]\subset(\mum\mup)$ and $\ga\in (0,\pi/2)$, there exist $r_0>0$ and $c>0$  such that for all $q\in \cD(0,r_0)$
and  $\xi\in i(\lm,\lp)+(\cC_\ga\cup(-\cC_\ga)\cup\{0\}$, 
\begin{enumerate}[(a)]
\item
\bbe\label{eq:boundPhiq_rec}
1-q\Phi(\xi)\not\in (-\infty,0] \quad \text{and}\quad |1-q\Phi(\xi)|\ge c;
\ee
\item\eq{phip1} and \eq{phim1} hold;
\item
$\phipq(\xi)$ (resp., $\phimq(\xi)$) admits analytic continuation to the lower half-plane 
with the cut $i(-\infty,\mum]$
(resp., to the upper half-plane 
with the cut $i[\mup,+\infty)$)
by
\beqa\label{phip1rec}
\phipq(\xi)&=&\frac{1-q}{(1-q\Phi(\xi))\phimq(\xi)},\\ \label{phim1rec}
\phimq(\xi)&=&\frac{1-q}{(1-q\Phi(\xi))\phipq(\xi)}.
\eqa

\end{enumerate}
\end{prop}
Assume that Conditions ($G;\be$) and ($\hG;a$) hold. We take $\omm<-\be$ such that both $\Phi$ and $\hG$ are analytic in $S_{[\omm,-\be]}$, and \eq{eq:boundPhiq_rec} is satisfied. Then, for $\eta\in\{\Im\eta=\omm\}$, we calculate   
\beqa\nonumber
\widehat{(\bfo_{[h,+\infty)}\phimq(D) G)}(\eta)&=&\int_h^{+\infty}dy\, e^{-iy\eta}\frac{1}{2\pi}\int_{\Im\xi=-\be}
d\xi\, e^{iy\xi}e^{-ia\xi}\phimq(\xi)\hG_0(\xi)\\\label{hV12}
&=&\frac{1}{2\pi}\int_{\Im\xi=-\be}\frac{e^{-i\eta h}}{i(\eta-\xi)}e^{i(h-a)\xi}\phimq(\xi)\hG_0(\xi)d\xi.
\eqa
(It is easy to see that the double integral converges absolutely, hence, the Fubini theorem is applicable.)
Substituting \eq{hV12} into  \eq{eq:V1Gh2}, we obtain
\beqa\label{eq:V1Gh3}
 V_1(G,h;n,x)
 &=&\frac{1}{2\pi i}\int_{|q|=r}dq\, q^{-n}\frac{1}{2\pi}\int_{\Im\eta=\omm}d\eta\, \frac{e^{i(x-h)\eta}\Phi(\eta)}{(1-q\Phi(\eta))\phimq(\eta)}\\\nonumber
 &&\times \frac{1}{2\pi}\int_{\Im\xi=-\be}\frac{e^{i(h-a)\xi}}{i(\eta-\xi)}\phimq(\xi)\hG_0(\xi)d\xi.
 \eqa
 
 \subsection{Representations of the Wiener-Hopf factors and the RHS of \eq{eq:V1Gh3} in terms of the Hilbert transform}\label{f-Hilbert}
The formulas in this subsection  are valid for wide classes of functions $\Phi$  (not necessarily analytic in a strip) and pdo operators $1-q\Phi(D)$; the functions may be unrelated to
Probability. See e.g., \cite{eskin}.
For $\xi$ on the line $\{\Im\xi=-\be\}$, we pass to the limit $\omm\uparrow -\be$ in \eq{phip1} and $\omp\downarrow -\be$ 
in \eq{phim1} and obtain
 \beqa\label{phip1vp}
 \phipq(\xi)&=&\left(\frac{1-q}{1-q\Phi(\xi)}\right)^{1/2}\exp\left[-\frac{1}{2\pi i}\mathrm{v.p.}
 \int_{\Im\eta =-\be}\frac{\xi\ln((1-q)/(1-q\Phi(\eta)))}{\eta(\xi-\eta)}d\eta\right]\\
 \label{phim1vp}
 \phimq(\xi)&=&\left(\frac{1-q}{1-q\Phi(\xi)}\right)^{1/2}\exp\left[\frac{1}{2\pi i}\mathrm{v.p.}
 \int_{\Im\eta =-\be}\frac{\xi\ln((1-q)/(1-q\Phi(\eta)))}{\eta(\xi-\eta)}d\eta\right]
 \eqa
 In \eq{eq:V1Gh3}, we change the order of integration w.r.t. $\eta$ and $\xi$, then pass to the limit $\omm\uparrow -\be$,
 and obtain 
 \beqa\label{eq:V1Gh3vp}
&& V_1(G,h;n,x)\\\nonumber
 &=&\frac{1}{2\pi i}\int_{|q|=r}dq\, q^{-n}\left\{\frac{1}{4\pi}\int_{\Im\xi=-\be}\frac{e^{i(x-a)\xi}\Phi(\xi)}{1-q\Phi(\xi)}\hG_0(\xi)d\xi\right.
 \\\nonumber
&&+
  \frac{1}{2\pi}\int_{\Im\xi=-\be}d\xi\,e^{i(h-a)\xi}\phimq(\xi)\hG_0(\xi)
  \frac{1}{2\pi i}\mathrm{v.p.}\int_{\Im\eta=-\be}d\eta\, \frac{e^{i(x-h)\eta}\Phi(\eta)}{(1-q\Phi(\eta))\phimq(\eta)(\eta-\xi)}.
 \eqa
 Formulas \eq{phip1vp}, \eq{phim1vp} and \eq{eq:V1Gh3vp} allows one to apply the fast Hilbert transform but are inefficient for
 numerical realizations due to the low decay of the integrands. See \cite{EfficientDiscExtremum}. Below, we use these formulas to
 justify the efficient $Z$-inversion in the case of symmetric L\'evy processes.  
 
\subsection{The case of symmetric L\'evy processes}\label{ss:bar_symm_Levy}
Assume that $\be$ in Condition ($G;\be$) can be chosen so that $\Phi(\xi-i\be)=\Phi(-\xi-i\be), \forall\ \xi\in\bR$. Then
the integrands on the RHS' of \eq{phip1vp}, \eq{phim1vp}, the first term on the RHS of \eq{eq:perp_barr_3} and the RHS of \eq{eq:V1Gh3vp} are analytic functions of $q\in\bC\setminus [1/\Phi(-i\be),+\infty)$ because $\Phi(\xi), \Phi(\eta)>0$
for  $\xi,\eta\in\bR$. The function  $\tV(q,x)=(I-qp_{(-\infty,h_+)}\Phi_\be(D)e_{(-\infty,h_+)})^{-1}G(x)$ is an analytic function of  
$q\in\bC\setminus [1/\Phi(-i\be),+\infty)$, hence, $\tV(q,x)$ is given by \eq{eq:perp_barr_3} and \eq{eq:V1Gh3vp}
for all $q\in\bC\setminus [1/\Phi(-i\be),+\infty)$. We make the sinh-deformation of the contour of integration in
the standard $Z$-inversion formula, choose the parameters of the deformation and simplified trapezoid rule using the general
results of Section \ref{s:Levy_symmetric}, and then evaluate each term in the simplified trapezoid rule making the sinh-change of variables
in the integrals w.r.t. $\eta$ and $\xi$, as in \cite{EfficientDiscExtremum}.

 \subsection{The case of asymmetric L\'evy processes}\label{ss:suff_barrier_eff_Z}

  In the first case that we consider, we need two contours $\cL(\de_\pm, u_\pm)\in \cT_{\mathrm{ext}}$, $\cL(\de_-,u_-)\prec \cL(\de_+;u_+)$, in the $\eta$-  and $\xi$--spaces in the pricing formula. We use the same contours $\cL(\de_+;u_+)$ (resp., $\cL(\de_-,u_-)$) to calculate the Wiener-Hopf factors on $\cL(\de_-,u_-)$ (resp., $\cL(\de_+,u_+))$.

\begin{lem}\label{lem:bounds_phipmq01} Let Conditions $(\psi,\infty)$  and $(\bV,\infty)$ hold. Then
\begin{enumerate}[(a)]
\item
For any  $r\in (0,1)$, there exist $\de^*\in(0,\pi/2)$ and $d>0$ such that, for any $\de_\pm\in (-\de^*,\de^*)$, $u_\pm\in (-\be-d,-\be+d)$, $q\in \cD(0, re^{\psi(-i\be)})$ and $\cL(\de_\pm, u_\pm)\in \cT_{\mathrm{ext}}$ such that
$\cL(\de_-,u_-)\prec \cL(\de_+,u_+)$, the following statements hold:
\begin{enumerate}[(i)]
\item
for any $\xi$ above $\cL(\de_-,u_-)$, 
\bbe\label{phipq_def_p}
\phipq(\xi)=\exp\left[-\frac{1}{2\pi i}\int_{\cL(\de_-,u_-)}\frac{\xi\ln((1-q)/(1-qe^{-\psi(\eta)}))}{\eta(\xi-\eta)}d\eta\right];
\ee
\item
for any $\xi$ below $\cL(\de_+,u_+)$, 
\bbe\label{phimq_def_m}
\phimq(\xi)=\exp\left[\frac{1}{2\pi i}\int_{\cL(\de_+,u_+)}\frac{\xi\ln((1-q)/(1-qe^{-\psi(\eta)}))}{\eta(\xi-\eta)}d\eta\right];
\ee

\item
for $\xi$ in the minimal simply connected region containing
$\{\Im\xi=-\be\}$ and $\cL(\de_\pm ,u_\pm)$, 
\bbe\label{eq:boundphimp_reciprocal1}
 |\phi^\pm_q(\xi)|+|1/\phi^\pm_q(\xi)|\le C,
 \ee
 where $C$ is independent of $q$ and $\xi$.
 \end{enumerate}
 \item
 For any $r\in (0,1)$ and $\ga\in (0,\pi)$, there exist $\de^*\in (0,\ga)$,  and $d>0$  such that
 for any $\de_\pm\in (-\de^*,\de^*)$, $u_\pm\in (-\be-d,-\be+d)$, $q\in \cD(0, re^{\psi(-i\be)})\cup (-\cC_{\pi-\ga})$ and $\cL(\de_\pm, u_\pm)\in \cT_{\mathrm{ext}}$ such that
$\cL(\de_-,u_-)\prec \cL(\de_+,u_+)$, the following statements hold:
\begin{enumerate}[(i)]
\item
for any $\xi$ below $\cL(\de_+,u_+)$, $\phimq(\xi)$ is given by \eq{phimq_def_m};
\item
for any $\xi$ above $\cL(\de_-,u_-)$, $\phipq(\xi)$ is given by \eq{phipq_def_p};
\item
for any $B>3$, there exist $C,c>0$ such that for $\xi\in \cL(\de_\pm ,u_\pm)$, 
\bbe\label{eq:boundphimp_reciprocal2}
 |\phi^\pm_q(\xi)|+|1/\phi^\pm(\xi)|\le C(1+|q|)^B(1+|\xi|)^{c\de^*},
 \ee
 where $c, C>0$ are independent of $q\in \cD(0, re^{\psi(-i\be)})\cup (-\cC_{\pi-\ga}), \xi\in \cL(\de_\pm ,u_\pm)$, and $c$ is independent of
 $\de^*$ as well. 
 \end{enumerate}
\end{enumerate}
\end{lem}
\begin{proof} (a) It is proved in \cite{EfficientDiscExtremum} that (1) for any $r\in (0,1)$,
there exist  $\ga'\in (0,\pi/2)$,  $C>0$, $d>0$  such that the statements (i)-(ii) hold for any $q\in \cD(0, re^{\psi(-i\be)})$ and any
pair of regular curves $\cL^\pm\subset i(-\be-d,-\be+d)+(\cC_{\ga'}\cup (-\cC_{\ga'})\cup\{0\})$ such that $\cL^-\prec\cL^+$,
 and (2)
part (iii) is valid for $\xi\in i(-\be-d,\be+d)+(\cC_{\ga'}\cup (-\cC_{\ga'})\cup\{0\})$. It remains to note that if  $\de^*>0$ is sufficiently small, then
$\cL(\de_\pm,u_\pm)\subset i(-\be-d,\be+d)+(\cC_{\ga'}\cup (-\cC_{\ga'})\cup\{0\})$.

(b) 
 Fix $r$, $\de^*$, $\de_\pm, u_\pm,\ga$ satisfying the conditions in (b),
 introduce $\cU_\pm=(\cD(0, re^{\psi(-i\be)})\cup(-\cC_{\pi-\ga}))\times \cL(\de_\pm,u_\pm)$, and  denote by $b^\pm_q(\xi)$ the  integrals under the exponential signs in \eq{phipq_def_p} and \eq{phimq_def_m}.
It suffices to prove the bound
 \bbe\label{eq:bound_Re_b_pm}
 |\Re b^\pm_q(\xi)|\le C+B\ln(1+|q|)+c\de^*\ln(1+|\xi|),\  (q,\xi)\in \cU_\pm,
 \ee
 where $B>3$ is arbitrary, $C, c>0$ are independent of $(q,\xi)\in \cU_\pm$, and $c$ is independent of $\de^*$ as well;
 $C$  depends on $B$ and $\de^*$.  The rather technical proof of \eq{eq:bound_Re_b_pm} is relegated to Section \ref{proof:lem:bounds_phipmq}. 
 \begin{rem}\label{rem:asym_contour}{\rm
  The last term on the RHS of \eq{eq:bound_Re_b_pm} and the last factor on the RHS of \eq{eq:boundphimp_reciprocal2} appear
 when we derive the bound in the region of $\xi$ large in absolute value. As it is seen from the proof, for these $\xi$, we can use the
 contours on the RHS' of \eq{phipq_def_p}-\eq{phimq_def_m}, symmetric w.r.t. $iu_\mp$ instead of the contours symmetric w.r.t.
 $i\bR$, 
 and then the last term does not appear. In Section \ref{proof:lem:bounds_phipmq}, we will indicate the place in the proof
 where additional asymmetric contours need to be used to eliminate the term $C_2\de^*\ln(1+|\xi|)$.
 
 }
 \end{rem}
 \end{proof}

\begin{thm}\label{thm:suff_barrier_1}
Let the following conditions hold:
 \begin{enumerate}[(i)]
\item
Conditions $(\psi,\infty)$,  $(\bV,\infty)$, $(G;\be)$ and Condition $(\hG;a)$ with $a<h$; 
\item either $(\barnu+1)/2\le \nu_0$ and $p(\de_\pm)\ge 0$ or  $\barnu+1\le \nu_0$ or $\hG_0(\xi)$ 
decays exponentially as $\xi\to \infty$ in a cone around the line $\{\Im\xi=-i\be\}$.

 \end{enumerate}  

Then
\begin{enumerate}[(a)]
\item
for the same $r$, $\de^*$ and $\cL(\de_\pm,u_\pm)$ as in Lemma \ref{lem:bounds_phipmq01} (a), 
\beqa\label{eq:V1Gh6}
 V_1(G,h;n,x)
 &=&\frac{1}{2\pi i}\int_{|q|=re^{\psi(-i\be)}}\frac{dq}{q^{n}} \frac{1}{2\pi}\int_{\cL(\de_-,u_-)}d\eta\, \frac{e^{i(x-h)\eta}\Phi(\eta)}{(1-q\Phi(\eta))\phimq(\eta)}\\\nonumber
 &&\times \frac{1}{2\pi}\int_{\cL(\de_+,u_+)}\frac{e^{i(h-a)\xi}}{i(\eta-\xi)}\phimq(\xi)\hG_0(\xi)d\xi;
 \eqa
\item in addition, let $n>7$. 
Then,
for any $r\in(0,1)$ and $\ga\in (0,\pi/2)$, and $\om_\ell$ and $d_\ell$ satisfying
$\om_\ell-d_\ell>\ga/2$ and $\om_\ell+d_\ell<\pi/4$, there exist
$\de^*\in (0,\ga)$ and $d>0$,
$\sg_\ell\in \bR$ and $b_\ell>0$ such that 
\begin{enumerate}[(i)]
\item
$\chi_{L; \sg_\ell, b_\ell, \om_{\ell}}(0)=re^{\psi(-i\be)}$ and 
$\cL_{L;\sg_\ell, b_\ell, \om_{\ell}}\subset \cD(0,r')\cup (-\cC_{\pi-\ga})$, where $r'\in (r,1)$; 
\item
for any pair of the contours  $\cL(\de_\pm,u_\pm)$  from Lemma \ref{lem:bounds_phipmq01} (b),
\beqa\label{eq:V1Gh7}
 V_1(G,h;n,x)
 &=&\frac{1}{2\pi i}\int_{\cL_{L;\sg_\ell,b_\ell,\om_\ell}}\frac{dq}{q^{n}}\frac{1}{2\pi}\int_{\cL(\de_-,u_-)}d\eta\, 
 \frac{e^{i(x-h)\eta}\Phi(\eta)}{(1-q\Phi(\eta))\phimq(\eta)}\\\nonumber
 &&\times \frac{1}{2\pi}\int_{\cL(\de_+,u_+)}\frac{e^{i(h-a)\xi}}{i(\eta-\xi)}\phimq(\xi)\hG_0(\xi)d\xi.
 \eqa
\end{enumerate}
\end{enumerate}
\end{thm}

\begin{proof} (a)    
It is proved in \cite{EfficientDiscExtremum} that, for any $r\in (0,1)$,
there exist  $\ga'\in (0,\pi/2)$,  $C>0$, $d>0$  such that \eq{eq:V1Gh6} holds for any $q\in \cD(0, re^{\psi(-i\be)})$ and any
pair of regular curves $\cL^\pm\subset i(-\be-d,-\be+d)+(\cC_{\ga'}\cup (-\cC_{\ga'})\cup\{0\})$ such that  $\cL^-$ is below $\cL^+$, 
the wings of $\cL^-$ point down, and the wings of $\cL^+$ point up. In \cite{EfficientDiscExtremum},
the marginally different formula stemming from \eq{eq:V1Gh} was used,  the contours were of a different form.
The proof is based on a bound for the inner double integrand. The bound follows from \eq{eq:boundphimp_reciprocal1}, Conditions ($G;\be$) and ($\hG;a$), and the following two properties of the contours $\cL^\pm$: (1) the distance between the two contours is positive; (2) the oscillating factors decay as $\eta\to\infty$ along $\cL^-$, and $\xi\to\infty$ along $\cL^+$. For the contours $\cL(\de_\pm,u_\pm)$, 
we have two subtleties: the distance is zero if 
$1-\nu_0<0$ (the contours do not intersect but have the real line as the common asymptote), and if $\barnu+1-\nu_0>0$, then one of the oscillating factors
$e^{i(x-h)\eta}$ and $e^{i(h-a)\xi}$  increases (in absolute value) as an exponential of 
$|\eta|^{\barnu+1-\nu_0}$ and
 $|\xi|^{\barnu+1-\nu_0}$, respectively, and the other one decays. See Lemma \ref{lem:extended curves_0}.
 Using Conditions ($\psi,\infty$), ($G;\be$), ($\hG;a$) and Lemma \ref{lem:extended curves_0}, we derive the bound for the inner double integrand on the RHS of \eq{eq:V1Gh6} via $CF(\eta,\xi)$, where
 \[
 F(\eta,\xi)=\frac{e^{-p(\de_-)(x-h)|\eta|^{\barnu+1-\nu_0}(1+\al_1(\eta))}e^{-d_0|\eta|^{\nu_0}(1+\al_2(\eta))}
 e^{-p(\de_+)(h-a)|\xi|^{\barnu+1-\nu_0}(1+\al_3(\xi))}}{|\xi-\eta||\xi|},
 \]
and $\al_1(\eta),\al_2(\eta)\to 0$ as $\eta\to\infty$ and $\al_3(\xi)\to 0$ as $\xi\to\infty$ along $\cL(\de_-,u_-)$ and $\cL(\de_+,u_+)$,
respectively. Under condition $(\barnu+1)/2\le \nu_0$, the product of the first two factors in the denominator decays faster than
$e^{-(d_0/2)|\eta|^{\nu_0}}$ as $\eta\to \infty$, independently of the sign of $p(\de_-)$. The third factor is uniformly bounded  if either $p(\de_+)\ge 0$ or $\barnu+1-\nu_0\le 0$. Hence, under either of these two conditions, $F(\eta,\xi)$ admits a bound
\[
F(\eta,\xi)\le C_N(1+|\eta|)^{-N}|\eta-\xi|^{-1}(1+|\xi|)^{-1},
\]
for any $N$, where $C_N$ is independent of $(\eta,\xi)$. In the region $\{(\xi,\eta)\ |\ |\xi|>2(|\eta|+1)\}$, 
\[
F(\eta,\xi)\le 2C_N(1+|\eta|)^{-N}(1+|\xi|)^{-2},
\]
and the RHS is absolutely integrable if $N>1$, and the integral w.r.t. $\xi$ over $\{\xi\ |\ |\xi|\le 2(|\eta|+1)\}$ is bounded by
\[
C_{1N}\int_0^{+\infty}(1+|\eta|)^{-N+1+\nu_0-1}d\eta.
\]
This integral converges if $N>\nu_0+1$. Since $\de^*>0$ and $d>0$ are arbitrary, the same argument shows that the inner double integral
is absolutely convergent in the process of the deformation of the initial contours $\{\Im\eta=\omm\}$ and $\{\Im\xi=-\be\}$ into $\cL(\de_-,u_-)$ and $\cL(\de_+,u_+)$, respectively, and is bounded by a constant independent of $q\in \cD(0, re^{\psi(-i\be)})$.

Finally, if $\barnu+1-\nu_0>0$ and $p(\de_+)<0$ but $\hG_0(\xi)$ exponentially decays, then the inner double integrand on the 
RHS of \eq{eq:V1Gh6} decays faster than any polynomial of $(\eta,\xi)$, 
with the constant in the bound independent of $q\in \{|q|\le r_0\}$, hence, the deformation is justified.
This finishes the proof of (a).

(b) Now we deform the outer contour in \eq{eq:V1Gh6} using Lemma \ref{lem:bounds_phipmq01} (b). The justification 
of the absolute convergence of the inner double integral in the process of deform is essentially the same as in (a).
The additional important element is the remark that, in the process of deformation,
$1-q\Phi(\eta)\not\in (-\infty,0]$ and $1/(1-q\Phi(\eta))$ is uniformly bounded.\end{proof}

Now we consider the case 
 $\nu_0\le (\barnu+1)/2$. We modify the statement and  proof of
  Theorem \ref{thm:suff_barrier_1} in the same vein as 
  Theorem \ref{thm:Euro-non-symm2} modifies Theorem \ref{thm:Euro-non-symm1}. Additional subtleties are
  as follows.
  \begin{enumerate}[(1)]
  \item
  Instead of one contour $\cL^*(x^*,\de,u)$ in Theorem \ref{thm:Euro-non-symm2}, we need two contours $\cL^*(x^*;\de_\pm,u_\pm)$, with 
  $\de_-<0<\de_+$ and $u_-<-\be<u_+$;
  \item
  We use \eq{eq:boundphimp_reciprocal2} for $\xi\in \cL^*(x^*, \de_\pm ,u_\pm)$. In the proof of the modified bound
  \eq{eq:boundphimp_reciprocal2}, we use \eq{phimq_def_m}  and \eq{phipq_def_p}, without  changing the contours of integration. Since $\de_-<0<\de_+$ and $|x|^{1-\nu_0}$
 increases as $x\to\pm\infty$, the proof of \eq{eq:boundphimp_reciprocal2} is repeated verbatim.
  \end{enumerate}

  \begin{thm}\label{thm:suff_barrier_3}  Let Conditions $(\psi,\infty)$, $(\bV,\infty)$, ($G;\be$), ($\hG;a$) hold, and
$\nu_0\in (0,1)$. 
Then the statements of Theorem \ref{thm:suff_barrier_1} hold with $\cL^*(x^*;\de_\pm,u_\pm)$
in place of $\cL(\de_\pm,u_\pm)$.

\end{thm}

 \section{Conclusion}\label{s:concl}
 In the paper, we derived several sets of sufficient conditions for applicability of the new efficient numerical realization of
 the inverse $Z$-transform, introduced in \cite{EfficientDiscExtremum} and applied  to pricing of barrier and lookback options with discrete monitoring, in L\'evy models. The idea is to transform the circle of integration in the standard $Z$-inversion formula into an infinite contour such that, after an appropriate sinh-change of variables, the integrand is analytic in a strip around the line of integration and the simplified trapezoid rule with a moderate number of terms allows one to satisfy the error tolerance of the order of $E-15$, whereas the trapezoid rule applied to the initial integral may require 100 times more terms.
 
We analyzed the complexities of the two schemes in detail in applications to evaluation of powers $P^n$ of a bounded operator
$P$ acting in a Banach space. The key condition for the applicability of the efficient $Z$-transform is that the spectrum $\sg(P)$ of $P$ is in the right half-plane. If $P$ is sectorial with the opening angle $2\ga$, where $\ga\in (0,\pi/2)$,
the efficiency increases as $\ga\to 0$. We derived an explicit prescription for an approximately optimal choice of the parameters of the scheme given the error tolerance. The limiting case $\ga=0$ arises in applications to pricing options and calculation expectations of various kind in symmetric L\'evy model (after an appropriate Esscher transform, the characteristic exponent 
satisfies $\psi(\xi)=\psi(-\xi)$.) We derived efficient pricing formulas and numerical algorithms for European options in the symmetric case. The symmetry condition implies that not only the infinitesimal generator of the process is sectorial but
the transition operators are sectorial as well.  In Section \ref{s:Euro_Levy_non-symmetric}, we consider pricing European options in non-symmetric L\'evy models. We construct a non-linear deformation of the dual space which makes
the transition operator sectorial, with an arbitrary small opening angle, under mild conditions that are satisfied 
for wide classes of
 non-symmetric Stieltjes-L\'evy models (SL-processes). The class of SL-processes is introduced in \cite{EfficientAmenable},
  and the general definition in terms of Stieltjes measures is the key point of the construction of the deformation. 
  It is demonstrated in \cite{EfficientAmenable} that essentially all popular classes of L\'evy processes are SL-processes. 
  
  Additional conditions on the parameters are satisfied for processes with non-zero Brownian motion component, with drift, and jump component which is a mixture of several pure jump processes. Each infinite variation component should be symmetric after an appropriate Esscher transform which can be different for different components. If there is no BM component, then additional conditions on the drift and jump components of lower orders should be imposed. 
  
In Section \ref{s:barrier_Levy_non_symm}, we consider in detail pricing  single barrier options in non-symmetric L\'evy models. The proof of the applicability of the efficient inverse $Z$-transform to single barrier options is based on the bounds for the Wiener-Hopf factors in the pricing formula, a pair of contours in $\bC_\eta$ and $\bC_\xi$ planes being used. The contours are different from the 
contours in \cite{EfficientDiscExtremum} for options with barrier-lookback features and in the iterative scheme in 
double barrier options in \cite{EfficientDiscDoubleBarrier}. In the present paper, we considered single barrier options
with the payoff functions generalizing puts, calls and digitals with the strike in the inaction region; exactly the same modified contours
can be used to justify  the application of the efficient inverse $Z$-transform to no-touch options, credit default swaps and more general  options with barrier-lookback features. The same technique can be very efficient for evaluation of the
various insurance products in the long run.


\appendix 

\section{Technicalities}\label{s:tech}

\subsection{Proof of \eq{eqpm:H_trap}}\label{Proof of eq:bound_H_trap}
As $r\uparrow 1$, 
\[
H(r)\sim2\int_0^{\pi/2}(1-2r\cos \varphi+r^2)^{-1/2}  d\varphi=2\int_0^{\pi/2}((1-r)^2+r\varphi^2+O(\varphi^4))^{-1/2}  d\varphi.
\]
We take a function $\de:(0,1)\to \bR_+$ such that $\de(r)\to 0$ and $(1-r)/\de(r)\to 0$ as $r\uparrow 1$,
and consider the integrals $I_1, I_2, I_3$ over the subsets $U_1=[0,(1-r)/\sqrt{r}]$,
$U_2=((1-r)/\sqrt{r},\de(r))$ and $U_3=[\de(r),\pi/2]$, respectively. On $U_1$, the integrand admits an upper bound via $C(1-r)$, hence, $I_1=I_1(r)\le C_1$, where $C, C_1$ are independent of $r$. On $U_3$, the integrand is bounded by $C\varphi^{-1}$,
hence, $U_3=O(\ln(1/\de(r))=o(-\ln(1-r))$ as $r\to 1$. Finally, on $U_2$, the integrand can be represented as
$
(1+O(\de(r)^2)((1-r)^2+r\varphi^2))^{-1/2}. $ Therefore, letting $\de_1(r)=\de(r)\sqrt{r}/(1-r)$, and making the change of variables
$\varphi=y(1-r)/\sqrt{r}$, we obtain
\beqast
I_2(r)&\sim&2\int_{(1-r)/\sqrt{r}}^{\de(r)}((1-r)^2+r\varphi^2)^{-1/2}d\varphi          \\
&=&2r^{-1/2}[\ln(\de_1(r)+\sqrt{\de_1(r)^2+1}-\ln(1+\sqrt{2})]\sim -2\ln(1-r),
\eqast
and \eq{eqpm:H_trap} follows.


\subsection{Proof of Lemma \ref{lem:vect field}}\label{proof_lem:vect field}
 By the Cauchy-Riemann conditions, the vector  field $\bV$ degenerates at $\xi$ if and only if $\xi$ is a zero of $\psi'(\xi)$; under condition \eq{eq:as_exp_psi}, for any $\eps>0$, the number of zeros of $\psi'(\xi)$ on $\{|\Re \xi|>\eps\}$ is finite. Hence, the number of zeros on
$\bC\setminus i\bR$ is at most countable, hence, for almost all $\de$, any $\cL\in \cT$ satisfying $\Im\psi(\xi)=\de, \xi\in \cL$, does not tend to any zero.
For such an $\cL$ and any compact $K\subset \bC\setminus i\bR$, there exists $C>0$ such that the set $\bV(\cL\cap K)$ is bounded away from 0. In view of \eq{eq:as_exp_psi}, it remains to prove that almost all $\cL$ are bounded away from $i\bR$.
 Consider $iy\in i\bR$. If $\psi(iy)\in\bR$, $\cL\in \cT$ cannot tend to $iy\in i\bR$. 
Next, the number of atoms of the measures $\cG^0_\pm$ is at most countable,
as well as the number of the end points of intervals on which one of the measures is absolutely continuous with the positive density.
Hence, it remains to prove that if $\cG^0_+(dt)\vert_{(a, b)}=c(t)dt$ (resp., $\cG^0_-(dt)\vert_{(a, b)}=c(t)dt$) where $c$ is a positive measurable
function and $0<a<b$,
then $i(-b,-a)$ (resp., $i(a,b)$) is a repulsive variety for any trajectory $\cL\in \cT$. Since
$\overline{\psi(\xi)}=\psi(-\bar\xi)$, it suffices to prove that $i(a,b)$ is repulsive. 
Let  $y_0\in (a,b)$. First, we prove that if there exists $\eps\in (0,1)$ such that $c(y)=c>0$ is constant on $(y_0-\eps, y_0+\eps)$, then there exists $x_0>0$ and $C(y_0)>0$ independent of $c$ and $\eps$
such that 
\beqa\label{der_psi_y} 
\dd_y \Im\psi(x+iy)\vert_{y=y_0}&\ge& cC(y_0),\ x\in (0,x_0),\\\label{der_psi_y_m}
\dd_y \Im\psi(x+iy)\vert_{y=y_0}&\le& -cC(y_0),\ x\in (-x_0,0).
\eqa
Let $\xi=x+iy$. We  differentiate \eq{eq:sSLrepr} w.r.t. $y$ 
\beqa\label{eq:sSLrepr_der}
\dd_y\psi(\xi)&=&\int_{(0,+\infty)}\frac{2a^+_2i\xi+a^+_1}{t-i\xi}\cG^0_+(dt)-\int_{(0,+\infty)}\frac{a^+_2\xi^2-ia^+_1\xi}{(t-i\xi)^2}\cG^0_+(dt)
\\\label{eq:sSLrepr_der2}
&&+\dd_y[a^-_2\xi^2+ia^-_1\xi)ST(\cG^0_-)i\xi]
+\sg^2i\xi+\mu. 
\eqa
For $y>0$, the imaginary parts of the terms on the RHS of \eq{eq:sSLrepr_der} and
the sum of the last two terms on the RHS \eq{eq:sSLrepr_der2} are $O(|x|)$ as $x\to 0$. Hence, it remains to prove
\eq{der_psi_y}-\eq {der_psi_y_m} for the imaginary part of the derivative on the RHS of \eq{eq:sSLrepr_der2}. Since
\[
\dd_y\Im[(a^-_2\xi^2+ia^-_1\xi)ST(\cG^0_-)i\xi]=\Im\int_{y_0-\eps}^{y_0+\eps}\dd_y\frac{a^-_2\xi^2+ia^-_1\xi}{t+i\xi}c dt
+O(|x|),
\]
uniformly in $\eps, c$, and $a^-_j\ge 0$, it suffices to consider the integral above 
in cases (1) $a^-_2=0, a^-_1=1$ and (2) $a^-_2=1, a^-_1=0$.
In Case (1), we calculate
\[
\dd_y\left(\frac{i\xi}{t+i\xi}\right)=\dd_y\left(1-\frac{t}{t+i\xi}\right)=-\frac{t}{(t+i\xi)^2}=-\frac{1}{t+i\xi}+\frac{i\xi}{(t+i\xi)^2}
\]
and integrate
\beqast
c\int_{y_0-\eps}^{y_0+\eps}\left(\frac{i\xi}{t+i\xi}\right)'_y\left|_{\Im\xi=y_0}dt\right.&=&-c\int_{y_0-\eps}^{y_0+\eps}\left(\frac{1}{t-y_0+ix}
-\frac{-y_0+ix}{(t-y+ix)^2}\right)dt\\
&=&-c\left(\ln\frac{\eps+ix}{-\eps+ix}+(-y_0+ix)\left[\frac{1}{\eps+ix}-\frac{1}{-\eps+ix}\right]\right)\\
&=&c\left( \mathrm{sign}(x)i\pi+O(|x|)-(y_0-ix)\frac{2\eps}{\eps^2+x^2}\right),
\eqast
where the constant in the $O$-term is independent of  $\eps, c$. 
Since $2\eps|x|/(\eps^2+x^2)\le 1$, we have
\bbe\label{bound_dr1}
 \mathrm{sign}(x)\Im \int_{y_0-\eps}^{y_0+\eps}\dd_y\left(\frac{i\xi}{t+i\xi}\right)\vert_{y=y_0}dt\ge  \mathrm{sign}(x)( \pi-1)+O(|x|).
 \ee
 The calculations in Case (2) are similar but the result is marginally different from \eq{bound_dr1}:
 \bbe\label{bound_dr2}
 \mathrm{sign}(x)\Im \int_{y_0-\eps}^{y_0+\eps}\dd_y\left(\frac{\xi^2}{t+i\xi}\right)\vert_{y=y_0}dt\ge  \mathrm{sign}(x) 2y_0(\pi-1)+O(|x|).
 \ee
 Since $a^-_j\ge 0$, \eq{der_psi_y}-\eq{der_psi_y_m} follow from \eq{bound_dr1}-\eq{bound_dr2}.

To finish the proof, we omit the assumption that the density of $\cG^0_-$ is constant in a neighborhood of $y_0\in (a,b)$. Let $n_0$ satisfy
$a<y_0-1/n_0, y_0+1/n_0<b$. For $n\ge n_0$,
 modify the density $c$ on $(y_0-1/n, y_0+1/n)$ replacing $c$ with $c(y_0)$, denote the resulting measure by $\cG^0_{n,-}$, the new
 characteristic function by $\psi_n$, and the new set of trajectories by $\cT_n$. It is evident that $\psi_n$ converges to
 $\psi$ on any compact $K\subset \bC\setminus i\bR$.
 It follows from \eq{der_psi_y}-\eq{der_psi_y_m}  that no trajectory $\cL_{n}\in \cT_n$ s.t. $\Im\psi_n(\eta)\neq 0, \eta\in \cL_n,$ crosses the segment connecting  $(-x_0,y_0)$ and $(x_0,y_0)$,
 hence, no trajectory $\cL\in \cT$ s.t. $\Im\psi(\eta)\neq 0$, $\eta\in \cL$, passes between $(-x_0,y_0)$ and $(x_0,y_0)$.

  \subsection{Proof of \eq{eq:bound_Re_b_pm}}\label{proof:lem:bounds_phipmq}
We prove the bound \eq{eq:bound_Re_b_pm} for $b^+_q(\xi), (q,\xi)\in \cU_+, \Re\xi\ge 0;$ the proof for $\xi$ in the left half-plane and proof for $b^-_q(\xi), (q,\xi)\in \cU_-$
are by symmetry. 
We use the following proposition, which is immediate from the choice of $\de^*, \de_\pm, u_\pm$ and construction of 
 $\cL(\de_\pm,u_\pm)$.  
\begin{prop}\label{prop:facts}
 The curves  $\cL(\de_\pm,u_\pm)$ are symmetric w.r.t. the imaginary axis, and
\begin{enumerate}[(1)]
\item
$1-q\Psi(\xi)\not\in (-\infty,0]$ for all $(q,\xi)\in \cU_\pm$;
\item
there exists
$C>0$ such that $|1/(1-q\Phi(\xi))|\le C$ for all $(q,\xi)\in \cU_\pm$;
\item
for any $\eps>0$, the curves $\cL(\de_\pm,u_\pm)$ can be constructed so that 
$\Im\psi(\xi)=\de_\pm$ for $\xi\in\cL(\de_\pm,u_\pm)$ such that $\Re\xi\ge\eps$, and 
$\Im\psi(\xi)=-\de_\pm$ for $\xi\in\cL(\de_\pm,u_\pm)$ such that $\Re\xi\le -\eps$
\item
there exists $R_1>0$ such that if $\eta\in\cL(\de_\mp,u_\mp)$, $|\eta|>R_1$, then $\eta$ is of the form 
$\eta_\mp(x):=x+iy(\cL(\de_\mp,u_\mp); |x|)$.
  We can choose a smooth parametrization  $\bR\mapsto \cL(\de_\mp,u_\mp)$ introduced in Remark \ref{rem:param_cL} so that if $\eta\in \cL(\de_\mp,u_\mp)$ and $|x|\ge R_1$, 
  then $\eta=\eta_\mp(x):=x+iy(\cL(\de_\mp,u_\mp); |x|)$;
  \item
  there exists $Q_0>0$ such that, for $q\in (-\cC_{\pi-\ga})$ s.t. $|q|\ge Q_0$, and any $a\in [1/2,2]$,
  the equation $a-|q| e^{-\Re  \psi(\eta_-(|x|))}=0$ has a unique solution $x^*_{+}(a,q)$ on $(\Re\eta_-^{-1}(R_1),+\infty)$,
  and a unique solution $x^*_{-}(a,q)=-x^*_{+}(a,q)$ on $(-\infty,-\Re\eta_-^{-1}(R_1))$;
  \item there exists $\ka>0$ such that  $x^*_{\pm}(a,q)= \pm(\ln(|q|/(ad_0)))^{1/\nu_0}(1+(\ln|q|)^{-\ka})$ as $q\to \infty$;
  \item
  there exists $C>0$ such that for $x\in (x^*_{-}(2,q),x^*_{+}(2,q))$ ,
  \bbe\label{repr:ln(1-qPhi)_near0}
 | \ln(1-q\Phi(\eta_-(x)))-\ln(-q)-\ln\Phi(\eta_-(x)))|\le C/|q\Phi(\eta_-(x)))|;
 \ee
 \item
 there exists $C>0$ such that if $x>x^*_{+}(0.5,q)$ or $x<x^*_{-}(0.5,q)$,
  \bbe\label{repr:ln(1-qPhi)_far0}
 | \ln(1-q\Phi(\eta_-(x)))+q\Phi(\eta_-(x)))|\le C(|q\Phi(\eta_-(x))))^2;
 \ee
 \item
 there exists $\ka>0$ such that $(\eta_-)'(x)=1+O(|x|^{-\ka})$ as $x\to\pm\infty$.
 
\end{enumerate}
\end{prop}

\vskip0.1cm
\noindent
{\sc 
A bound on $\cU_+(Q_0)=\{(q,\xi)\in \cU_+\ |\ |q|\le Q_0\}$.} We may assume that $u_-\neq 0$. Then 
$b^+_q(\xi)=\bfo_{(0,+\infty)}(u_-)\ln(1-q)+b^{++}_q(\xi)$, where
\bbe\label{bpp}
b^{++}(\xi)=-\frac{1}{2\pi i}\int_{\cL(\de_-,u_-)}\frac{\xi\ln(1-q\Phi(\eta))}{\eta(\xi-\eta)}d\eta.
\ee
Thus, it suffices  to derive a bound for $b^{++}_q(\xi)$. 
Fix $\xi\in \cL(\de_+,u_+)$ and divide the contour 
$\cL(\de_-,u_-)$ on the RHS of \eq{bpp} into two contours by
 conditions $|\eta|\le |\xi|/2$ and $|\eta|> |\xi|/2$. On the first contour, the integrand is bounded by $C(Q_0)e^{-(d_0/2)|\eta|^{\nu_0}}/|\eta|$,
 and on the second one, by $C(Q_0)e^{-(d_0/2)|\eta|^{\nu_0}}|\eta|^{\nu_0-1}$, where $C(Q_0)$ is independent of $\xi\in \cL(\de_+,u_+)$.
 Thus, for $(q\xi)\in \cU(Q_0)$, the bound \eq{eq:bound_Re_b_pm} for $b^{++}_q(\xi)$ is valid with $B=0$, $C_2=0$ and $C_1$ that depends on $\bQ_0$ only. 
 
 \vskip0.1cm
\noindent
{\sc 
Bounds on
 $\cU\setminus \cU(Q_0)$.} It is convenient to prove the bound \eq{eq:bound_Re_b_pm} considering integrals over a finite number of subsets of $\cU_-$, for $\xi$ in specified subsets of $\cU_+$.  In the bounds below, $C,C_1,\ldots, c,c_1,\ldots $ denote constants that can be chosen
the same for all $q, |q|\ge Q_0$, and all $\xi$,  $\eta$ from the subsets of $\cL(\de_\pm,u_\pm)$ considered at each step.

 We may assume that $|q|$ is sufficiently large so that $x^*_\pm(a,q)$, $a\in [1/2, 2]$, are well-defined;
 $\bR\ni x\mapsto \eta_-(x)$ is the parametrization of $\cL(\de_-,u_-)$ defined in Proposition \ref{prop:facts} (4).
 Let $\eps>0$ be as in Proposition \ref{prop:facts} (3),
 and let $x_{*,+}(\eps)$  be the minimal $x>0$ such that $\Re\eta_-(x)=\eps$,  and $x_{*,-}(\eps)=-x_{*,+}$  the maximal $x<0$ such that $\Re\eta_-(x)=-\eps$. Introduce 
 \beqast
 J^0_{\eps}&=&\{x\in\bR\ |\ x_{*,-}(\eps)<x<x_{*,+}(\eps)\},\\
 J^{+,1}_{\eps, q}&=&\{x\in\bR\ |\ x_{*,+}(\eps)<x<x^*_{+}(2,q)\},
 J^{-,1}_{\eps, q}=\{x\in\bR\ |\ x^*_{-}(2,q)<x<x_{*,-}(\eps)\},\\
 J^{+,2}_{q}&=&\{x\in\bR\ |\ x^*_{+}(2,q)\le x\le x^*_{+}(1/2,q)\},
 J^{-,2}_{q}=\{x\in\bR\ |\ x^*_{-}(1/2,q)\le x \le x^*_{-}(2,q)\},\\
 J^{+,3}_{q}&=&\{x\in\bR\ |\  x> x^*_{+}(1/2,q)\},
 J^{-,3}_{q}=\{x\in\bR\ |\ x<x^*_{-}(1/2,q)\}.
 \eqast
 For a union of intervals $J$, define
 \beqa\label{eq:bpJ1}
 b^{+}_{q}(J;\xi)&=&-\frac{1}{2\pi i}\int_{J}\frac{\xi\ln\frac{1-q}{1-q\Phi(\eta_-(x))}}{\eta(\xi-\eta_-(x))}d\eta_-(x),\\\label{eq:bpJ1n}
 b^{+,n}_{q}(J;\xi)&=&-\frac{1}{2\pi i}\int_{J}\frac{\xi\ln(1-q)}{\eta(\xi-\eta_-(x))}d\eta_-(x),\\\label{eq:bpJ1d}
 b^{+,d}_{q}(J;\xi)&=&-\frac{1}{2\pi i}\int_{J}\frac{\xi\ln(1-q\Phi(\eta_-(x)))}{\eta(\xi-\eta_-(x))}d\eta_-(x),
 \eqa
and  note that
 \bbe\label{sum_bp}
 b^{+}_{q}(J;\xi)= b^{+,n}_{q}(J;\xi)- b^{+,d}_{q}(J;\xi).
 \ee
\vskip0.1cm
\noindent
{\sc 
Bound for  $b^+_q(J^0,\xi)$.}   The integrand being uniformly bounded by $C\ln|q|$,   we have
 \bbe\label{eq:bound_eps_bpp}
 | b^{+}_{q}(J^0;\xi)|\le C_1\eps_1 \ln|q|,
 \ee
 where $\eps_1\to 0$ as $\eps\to 0$.

 Considering the integrals over $J^{\pm,1}_{\eps,q}$ and $J^{\pm,j}_{q}, j=2,3$, we use
 $\Im\psi(\eta)=\pm \de_-$.
 \vskip0.1cm
\noindent
{\sc 
Bound for  $b^+_q(J^{+,1}_{\eps, q}; \xi)+b^+_q(J^{-,1}_{\eps, q}; \xi)$.}
 Using \eq{repr:ln(1-qPhi)_near0}, we obtain 
 \bbe\label{repr:b1pm}
 b^+_q(J^{\pm,1}_{\eps, q}; \xi)=b^{+,\pm}_{q;0}(\xi)+b^{+,\pm}_{q;1}(\xi),
 \ee
 where
 \bbe\\\label{eq:bpJ10}
 b^{+,\pm}_{q;0}(\xi)=-\frac{1}{2\pi i}\int_{J^{\pm,1}_{\eps, q}}\frac{\xi\ln\Phi(\eta_-(x))}{\eta(\xi-\eta_-(x))}d\eta_-(x),
 \ee
 and $b^{+,\pm}_{q;1}$ admits the bound
 \beqa\nonumber
 |b^{+,\pm}_{q;1}(\xi)|&\le& C_1 \int_{J^{\pm,1}_{\eps, q}}\left|\frac{\xi}{q\Phi(\eta_-(x)))\eta(\xi-\eta_-(x))}(\eta_-)'(x)\right|dx
 \\\label{eq:bpJ11}
 &\le &C_2 \int_{J^{\pm,1}_{\eps, q}}\left|\frac{\xi}{\eta(\xi-\eta_-(x))}(\eta_-)'(x)\right|dx.
 \eqa
In \eq{eq:bpJ10}, we integrate by parts using 
 \bbe\label{int_by_part}
 \frac{\xi d\eta}{\eta(\xi-\eta)}=d(\ln\eta-\ln(\eta-\xi))
 \ee
 and equalities
 \beqast
 \ln\Phi(\eta_-(x^*_{\pm}(2,q))&=&-\Re\psi(x^*_{+}(2,q)\mp i\de_-,\\
 -d\ln\Phi(\eta_-(x))&=&d\Re\psi(\eta_-(x)), |x|\ge x_{*,+}(\eps).
 \eqast
 We represent the result in the form 
 \bbe\label{eq:bpJ102a}
b^{+,+}_{q;0}(\xi)+b^{+,-}_{q;0}(\xi)=-g_0(q,\xi)-g_1(q,\xi), 
 \ee
 where 
 \bbe\label{eq:g1}
g_1(q,\xi)=\frac{1}{2\pi i}\int_{J^{+,1}_{\eps, q}\cup J^{-,1}_{\eps, q}}\ln\frac{\eta_-(x)}{\eta_-(x)-\xi}d\Re\psi(\eta_-(x)),
\ee
and
 \beqa\label{eq:g0}
g_0(q,\xi)&=&\frac{1}{2\pi i}\left.\ln\Phi(\eta_-(x))(\ln(\eta_-(x))-\ln(\eta_-(x)-\xi))\right|_{x_{*,+}(\eps)}^{x^*_{+}(2,q)}
\\\nonumber
&&+\frac{1}{2\pi i}\left.\ln\Phi(\eta_-(x))(\ln(\eta_-(x))-\ln(\eta_-(x)-\xi))\right|_{x^*_{-}(2,q)}^{x_{*,-}(\eps)}
\\\nonumber
&=&g_{00}(q,\xi)+g_{01}(q,\xi)+g_{02}(q,\xi)+g_{03}(q,\xi),
\eqa
where
\beqast\label{def_g00}
g_{00}(q,\xi)&=&-\frac{1}{2\pi i}\Re\psi(\eta_-(x^*_{+}(2,q)))\left[\ln\frac{\eta_-(x^*_{+}(2,q))}{\eta_-(x^*_{+}(2,q))-\xi}
-\ln\frac{\eta_-(-x^*_{+}(2,q))}{\eta_-(-x^*_{+}(2,q))-\xi}\right],\\\label{def_g01}
g_{01}(q,\xi)&=&\frac{\de}{2\pi}\left[\ln\frac{\eta_-(x^*_{+}(2,q))}{\eta_-(x^*_{+}(2,q))-\xi}
+\ln\frac{\eta_-(-x^*_{+}(2,q))}{\eta_-(-x^*_{+}(2,q))-\xi}\right],\\\label{def_g02}
g_{02}(q,\xi)&=&\frac{1}{2\pi i}\left[\Re\psi(\eta_-(x_{*,+}(\eps))\ln\frac{\eta_-(x_{*,+}(\eps))}{\eta_-(x_{*,+}(\eps))-\xi}
-\Re\psi(\eta_-(x_{*,-}(\eps))\ln\frac{\eta_-(-x_{*,+}(\eps))}{\eta_-(-x_{*,+}(\eps))-\xi}\right],
\\\label{def_g03}
g_{03}(q,\xi)&=&-\frac{\de}{2\pi }\left[\ln\frac{\eta_-(x_{*,+}(\eps))}{\eta_-(x_{*,+}(\eps))-\xi}
+\ln\frac{\eta_-(-x_{*,+}(\eps))}{\eta_-(-x_{*,+}(\eps))-\xi}\right].
\eqast
Consider the terms $g_{0j}(q,\xi), j=0,1,2,3$. Evidently, $g_{02}(q,\xi)$ and $g_{03}(q,\xi)$ admit bounds 
\beqa\label{bound_g02}
|\Re g_{02}(q,\xi)|&\le & C\eps,\\\label{bound_g03}
|\Re g_{03}(q,\xi)|&\le & C_1|\de|(1+\ln|\xi|),
\eqa
where $C$ is independent of $\de$. \footnote{The bound for $g_{03}(q,\xi)$ and a similar bound
for $g_{01}(q,\xi)$ below are the only places where
the terms $C|\de|\ln(1+|\xi|)$ appear. For $|\xi|$ in a bounded region, the terms $C|\de|\ln(1+|\xi|)$ can be omitted,
and if $\Re\xi>0$ is large, we can repeat the arguments below verbatim replacing the contour of integration
$\cL(\de_-,u_-)$ (which is symmetric w.r.t. the imaginary axis) with the contour symmetric w.r.t. $iu_-$.
For the new contour, $\Im\psi(\eta_-(x))=\de_-$ for all $x$, and, therefore, the plus sign in the square brackets in the formulas for $g_{01}$ and $g_{03}$ becomes the minus sign.   In the result, the $C|\de|\ln(1+|\xi|)$-term disappears.
}
Consider $g_{01}(q,\xi)$. In the region $|\xi|\le x^*_{+}(2,q))/2$, $g_{01}(q,\xi)$ is uniformly bounded, and in the region
$|\xi|\ge 2x^*_{+}(2,q)$, $g_{01}(q,\xi)$ admits the same bound as $g_{03}(q,\xi)$. In the region $x^*_{+}(2,q)/2\le |\xi|
\le 2x^*_{+}(2,q)$, we use  bounds $|x^*_{+}(2,q)-\xi|\ge cx^*_{+}(2,q)^{1-\nu_0}$
and $x^*_{+}(2,q)\le C\ln|q|$ to obtain
\bbe\label{bound_g01}
|\Re g_{01}(q,\xi)|< \frac{|\de_-|}{\pi}(1+\ln|\xi|)+C\ln\ln|q|.
\ee
 Finally, $\Re\psi(\eta_-(x^*_{+}(2,q)))=\ln(|q|/2)$, and the argument of the difference
in the square bracket in the formula for $g_{00}(q,\xi)$ is less than $\pi$ in absolute value if $|q|$ is large. Hence,
\bbe\label{bound_g00}
|\Re g_{00}(q,\xi)|< \frac{1}{2}\ln|q|+C\ln\ln|q|.
\ee
Substituting \eq{bound_g00}, \eq{bound_g01}, \eq{bound_g02} and \eq{bound_g03} into the rightmost part of \eq{eq:g0}, we obtain
\bbe\label{bound_g0}
|\Re g_{0}(q,\xi)|< \frac{1}{2}\ln|q|+C\de^*(1+\ln|\xi|)+C\ln\ln|q|.
\ee

Consider $g_1(q,\xi)$. In view of Proposition \ref{prop:facts} (4),  the measure $d\Re\psi(\eta_-(x))$ is positive on $[R_1,+\infty)$ and negative on $(-\infty,-R_1]$,
where $R_1$ can be chosen the same for all $q,\xi$. Therefore, there exists $C>0$ independent of $q,\xi$ such that
\bbe\label{Re_b00}
|\Re g_1(q,\xi)|\le C+\frac{1}{2\pi}\pi (\Re\psi(\eta_-(x^*_{-,+}(2,q)))+\Re\psi(\eta_-(x^*_{-,-}(2,q))).
\ee
Since $2=|q|e^{-\Re\psi(\eta_-(x^*_{-,\pm}(2,q)))}$, there exists $C_1$ independent of $q,\xi$ such that
 \bbe\label{eq:bound_Re_b_pm_00}
 |\Re g_1(q,\xi)|\le C_1+\ln|q|.
 \ee
 Using  \eq{eq:bpJ102a}, \eq{eq:bound_Re_b_pm_00} and \eq{bound_g0}, we obtain 
 \bbe\label{bound_bpm10}
|\Re(b^{+,+}_{q;0}(\xi)+b^{+,-}_{q;0}(\xi))|\le C_1+\frac{3}{2}\ln|q|+C_2\ln\ln|q|.
\ee
 Consider $b^{+,+}_{q,1}(q,\xi)$ in the following three cases: 
 $|\xi|\ge 2 x^*_{+}(2,q)$, $|\xi|\le R_1$ and $R_1\le |\xi|\le 2 x^*_{+}(2,q)$. If $|\xi|\ge 2 x^*_{+}(2,q)$, then the (absolute value of the) integrand on the RHS of \eq{eq:bpJ11}
 admits an upper bound via $C_3 x^{-1}$, hence, 
 \bbe\label{bound_b_aux}
 |\Re b^{+,+}_{q;1}(\xi)|\le C_4+C_3 \ln  x^*_{+}(2,q)\le C_5 \ln\ln |q|.
 \ee
Let $|\xi|\le R_1$. Since $\cL(\de_+,u_+)\cap \cL(\de_-,u_-)=\emptyset$, the RHS of \eq{eq:bpJ11} is bounded by a constant 
 $C(R_1)$ independent of $q$ and $\{\xi\ |\ |\xi|\le R_1\}$. Similarly, the integral over $\{|\eta_-|\le R_1\}$ is bounded by a constant $C(R_1)$ independent of $q$ and $\xi\in\cL(\de_+,u_+)$. Below, we consider $\{\xi\ |\ |\xi|\ge R_1\}$ and the integrals over $J^{+}_{q; R_1}:=\{x\ |\ R_1<x<x^*(2,q)\}$. If $R_1$ is sufficiently large, $\xi$ and $\eta_-(x), |x|\ge R_1,$ are of the form
 $\xi=x'+iy(\de_+,x')$ and $\eta_-(x)=x+iy(\de_-,x)$, where $y_x(\de,x)\to 0$ as $x\to\pm \infty$. Furthermore,  $(\eta_-)'(x)=1+o(1)$ as $x\to\pm\infty$. Hence, it suffices to derive bounds for
 \[
 g(J, q,\xi)=\int_J\left|\frac{x'+iy(\de_+,|x'|)}{(x+iy(\de_-,x))(x'-x+i(y(\de_+,|x'|)-y(\de_-,x)))}\right|dx,
 \]
where $J=J^{+}_{q; R_1}$. 
 Fix $C>0$. Since $J^{+}_{q; R_1}$  is a subset of  the union of intervals $J_{C,j}, j=1,2,3$, defined
 by the conditions $R_1\le x\le x'(1-C(x')^{-\nu_0})$; $x'(1-C(x')^{-\nu_0})\le x\le  x'(1+C(x')^{-\nu_0})$;
 $ x'(1+C(x')^{-\nu_0})\le x\le 2 x^*_{-,+}(2,q)$, respectively, it suffices to derive bounds for  $g(J_{C,j}, q,\xi), j=1,2,3.$
 On $J_{C,3}$, the integrand is bounded by $1/|x-x'|$, hence, 
 $
 |g(J_{C,3},q,\xi)|$ admits the bound \eq{bound_b_aux}.  On $J_{C,2}$, the integrand is bounded by $C_7/|x'|^{1-\nu_0}$ since
 $|y(\de_+,x)-y(\de_-,x)|\ge cx^{1-\nu_0}$.
Hence, $
 |g(J_{C,2},q,\xi)|\le (C_1/|x'|^{1-\nu_0})2C|x'|^{1-\nu_0}=2CC_1$. On $J_{C,1}$, the integrand is bounded by 
 $C_2x'/(x(x'-x)$. Integrating the fraction $x'/(x(x'-x))$, we obtain that  $
 |g(J_{C,1},q,\xi)|$ admits the bound \eq{bound_b_aux}.  
 Thus,
 \bbe\label{bound_bpm_1}
 |\Re b^{+,+}_{q;1}(\xi)|\le C+C_1\ln\ln|q|.
 \ee
 The same bound for $ \Re b^{+,-}_{q;1,-}(\xi)$ holds, and it is simpler to derive because, in the proofs of the bounds 
 for $|\Re b^{+,+}_{q;1}(\xi)|$, we can use $|x-x'|\ge c(x'+|x|)$.
  Using \eq{bound_bpm10}, \eq{bound_bpm_1} and \eq{repr:b1pm}, we obtain
\bbe\label{J1}
\left|\Re( b^+_q(J^{+,1}_{\eps, q}; \xi)+b^+_q(J^{-,1}_{\eps, q}; \xi)) \right|\le C+1.5\ln|q|+C_1\ln\ln|q| + C_2\de^*\ln(1+|\xi|),
\ee
 where $C,C_1, C_2>0$ are independent of $q,\xi$, and $C_1,C_2$ are independent of $\de^*$.

  \vskip0.1cm
\noindent
{\sc 
Bounds for  $ \sum_{j=2,3}(b^{+,n}_{q}(J^{+,j}_{q};\xi))+b^{+,n}_{q}(J^{-,j}_{q};\xi))$.} Using \eq{int_by_part}, 
we obtain
\beqast\label{sum_4}
&& \sum_{j=2,3}(b^{+,n}_{q}(J^{+,j}_{q};\xi)+b^{+,n}_{q}(J^{-,j}_{q};\xi))\\\nonumber
&=&
\frac{1}{2\pi i}\ln(1-q)\left\{\ln\frac{\eta_-(x^*_{+}(2,q))}{\eta_-(x^*_{+}(2,q))-\xi}-
\ln\frac{\eta_-(x^*_{-}(2,q))}{\eta_-(x^*_{-}(2,q)))+\xi}\right\}.
\eqast
In the regions $\{|\xi|\ge 2|\eta_-(x^*_{+}(2,q))|\}$ and $\{|\xi|\le 0.5|\eta_-(x^*_{+}(2,q))|\}$, the
function in the curly brackets is uniformly bounded, and in the region $\{0.5|\eta_-(x^*_{+}(2,q))|\le |\xi|\ 2|\eta_-(x^*_{+}(2,q))|\}$,
the function admits a bound via $\ln x^*_{+}(2,q)^{\nu_0}=O(\ln\ln |q|).$ It follows that 
\bbe\label{sum_4_b}
|\Re\sum_{j=2,3}(b^{+,n}_{q}(J^{+,j}_{q};\xi)+b^{+,n}_{q}(J^{-,j}_{q};\xi))|\le C+\ln|q|+C_1\ln\ln|q|.
\ee

\vskip0.1cm
\noindent
{\sc Bounds for $b^{+, d}_q(J^{\pm,2}_q, \xi)$.}
Since $|b^{+, d}_q(J^{\pm,2}_q, \xi)|$ admit an upper bound 
\[
|b^{+, d}_q(J^{\pm,2}_q, \xi)|\le C \int_{J^{\pm,2}_{q}}\left|\frac{\xi}{\eta_-(x)(\xi-\eta_-(x))}(\eta_-)'(x)\right|dx,
\]
essentially the same proof as the proof of the bound \eq{bound_bpm_1} for $|b^{+,\pm}_{q;1}(\xi)|$, that starts
with the similar bound \eq{eq:bpJ11},  gives
the bound 
\bbe\label{bound_bd2pm}
|b^{+, d}_q(J^{\pm,2}_q, \xi)|\le C+C_1\ln\ln |q|.
\ee

\vskip0.1cm
\noindent
{\sc Bound for $b^{+, d}_q(J^{+,3}_q, \xi)$.} For $\xi\in\{\xi\in\cL(\de_+,u_+)\ |\ \Re\xi>0, |\xi|\le x^*_+(1/2,q)/2\}$, we have a bound of the form
\beqast
|b^{+, d}_q(J^{+,3}_q, \xi)|&\le &C\int_{x^*_+(1/2,q)}^{+\infty}\left|\frac{qe^{-\psi(\eta_-(x))}}{\eta_-(x)^{1-\nu_0}}\eta'_-(x)\right|dx\\
&=& C\int_{x^*_+(1/2,q)}^{+\infty}|q|\frac{e^{-\Re\psi(\eta_-)}(\Re\psi(\eta_-(x)))'}{\Re\psi(\eta_-(x)))'x^{\nu_0-1}}dx\\
&\le& C_1\int_{x^*_+(1/2,q)}^{+\infty}|q|e^{-\Re\psi(\eta_-)}(\Re\psi(\eta_-(x)))'dx=C_1/2,
\eqast
where $C_1$ is independent of $q,\xi$. Here we used $(\Re\psi(\eta_-(x)))'\sim d_0\nu_0 x^{\nu_0-}$ as $x\to+\infty$.

For $\xi\in\{\cL(\de_+,u_+)\ |\ \Re\xi>0, |\xi|> x^*_+(1/2,q)/2\}$, we fix $\eps\in (0,1)$, and represent the interval of integration $J^{+,3}_q$ as the union of subsets $J^{+,3;j}_{q;\eps}\subset J^{+,3}_q, j=1,2,3$, where $J^{+,3;j}_{q;\eps}$, $j=1,2,3$,
are defined by 
conditions $x<|\xi|(1-\eps)$, $|\xi|(1-\eps)\le x\le |\xi|(1+\eps)$ and $x>|\xi|(1+\eps)$, respectively. Depending on $\xi$, one of $J^{+,3;j}_{q;\eps}, j=1,2$, or both can be empty. For $j=1,3$, we have a bound
\beqast
|b^{+, d}_q(J^{+,3;j}_{q;\eps}, \xi)|&\le &\frac{C_1}{\eps}\int_{x^*_+(1/2,q)}^{+\infty}|q|\frac{e^{-\Re\psi(\eta_-(x))}(\Re\psi(\eta_-(x)))'}{
\Re\psi(\eta_-(x)))' x}dx\le C_2|q|\\
&\le & \frac{C_2}{\eps 2 x^*_+(1/2,q)}=o(1), |q|\to \infty.
\eqast
Next, similarly to the bound for $\xi\in\{\xi\in\cL(\de_+,u_+)\ |\ \Re\xi>0, |\xi|\le x^*_+(1/2,q)/2\}$, we obtain that
$
|b^{+, d}_q(J^{+,3;2}_{q;\eps}, \xi)|$ is uniformly bounded. We conclude that
\bbe\label{bound_bd3pm}
|b^{+, d}_q(J^{+,3}_q, \xi)|\le C.
\ee
The proof of the same bound for $|b^{+, d}_q(J^{+,3}_q, \xi)|$ is similar but simpler because we can use the boiund $|\xi-\eta_-(x)|\ge c(|\xi|+|\eta_-(x)|)$. 
Gathering bounds \eq{bound_bd3pm}, \eq{bound_bd2pm}, \eq{sum_4_b}, \eq{J1}, \eq{bound_bpm10}, \eq{eq:bound_eps_bpp}
and taking \eq{sum_bp} into account, we obtain \eq{eq:bound_Re_b_pm}.

 \end{document}